\documentclass{article}

\usepackage[utf8]{inputenc}

\usepackage[a4paper, margin=1.2in, nohead]{geometry}
\usepackage{setspace}
\usepackage[psamsfonts]{amssymb}
\usepackage{amsthm,amsmath,amsbsy,mathabx}
\usepackage{bbm, bm}
\usepackage[dvipsnames]{xcolor}

\usepackage[shortlabels]{enumitem}
\RequirePackage[colorlinks,citecolor=blue,urlcolor=blue]{hyperref}

\usepackage{graphicx,placeins}

\newtheorem{conj}{Conjecture}%[section]
\newtheorem{thm}[conj]{\bf Theorem}

\newtheorem{cor}[conj]{\bf Corollary}

\newtheorem{lemma}[conj]{\bf Lemma}
\newtheorem{rem}[conj]{\bf Remark}

\newtheorem{cond}{\sc Condition}

\makeatletter
\def\saveenum{\xdef\@savedenum{\the\c@enumi\relax}}
\def\resetenum{\global\c@enumi\@savedenum}
\makeatother
\def\bar{\overline}

\def\eps{\varepsilon}

\def\Cc{\mbox{$\mathcal C$}}
\def\Ac{{\mathcal A}}
\def\Dc{\mbox{$\mathcal D$}}

\def\Hc{\mbox{$\mathcal H$}}
\def\Ec{\mbox{$\mathcal E$}}

\def\Gc{\mbox{$\mathcal G$}}

\def\Nc{\mbox{$\mathcal N$}}

\def\Tc{\mbox{$\mathcal T$}}

\def\Bb{{\mathbb B}}
\def\Cb{{\mathbb C}}
\def\Rb{{\mathbb R}}
\def\Db{{\mathbb D}}
\def\Eb{{\mathbb E}}
\def\Pb{{\mathbb P}}

\def\Vb{{\mathbb V}}

\def\contiguous{\triangleleft\kern-.20em\triangleright}

\def\weakly{\buildrel {w} \over \longrightarrow}

\def\0{{\mathbf 0} }
\def\1{{\mathbf 1} }

\def\x{ {\bf x}}
\def\a{ {\bf a}}
\def\b{ {\bf b}}
\def\e{ {\bf e}}
\def\y{ {\bf y}}
\def\z{ {\bf z}}

\def\u{ {\bf u}}
\def\v{ {\bf v}}

\def\X{ {\bf X}}

\def\Z{ {\bf Z}}
\def\U{ {\bf U}}
\def\W{ {\bf W}}

\def\mds{\medskip}

\newcommand{\var}{{\mathrm Var}}

\date{\today}

\author{
Alexis Derumigny\thanks{University of Twente, 5 Drienerlolaan, 7522 NB Enschede, Netherlands. a.f.f.derumigny@utwente.nl.}, 
Jean-David Fermanian\thanks{Ensae, 5, avenue Henry Le Chatelier, 91764 Palaiseau cedex, France. jean-david.fermanian@ensae.fr.
}}

\title{Conditional empirical copula processes and generalized dependence measures}

\begin{document}

\maketitle

\begin{abstract}
    We study the weak convergence of conditional empirical copula processes, when the conditioning event has a nonzero probability. The validity of several bootstrap schemes is stated, including the exchangeable bootstrap. We define general - possibly conditional -  multivariate dependence measures and their estimators. By applying our theoretical results, we prove the asymptotic normality of some estimators of such dependence measures.
\end{abstract}

\mds


\mds

{\bf MCS 2020:} Primary: 62G05, 62G30; Secondary: 62H20, 62G09.

\mds

\section{Introduction}\label{sect-Introduction}

Since their formal introduction by Patton in~\cite{patton2006a,patton2006b}, conditional copulas have become key tools to describe the dependence function between the components of a random vector $\X:=(X_1,\ldots,X_p)\in \Rb^p$, given that a random vector of covariates $\Z:=(Z_1,\ldots,Z_q)\in \Rb^q$ is available. This concept, generalized in~\cite{fermanian2012}, may be stated as an extension of the famous Sklar's theorem: for every borelian subset $A\subset \Rb^q$ and every vectors $\x\in \Rb^p$, the conditional joint law of $\X$ given $(\Z\in A)$ is written
\begin{equation}
    F(\x|A) := \Pb \big(\X\leq \x | \Z\in A \big)= C_{\X|\Z}\big( \Pb(X_1\leq x_1 | \Z\in A),\ldots,\Pb(X_p\leq x_p | \Z\in A)
    \, \big|\, \Z\in A   \big),
    \label{def_cond_cop}
\end{equation}
for some random map $C_{\X|\Z}(\cdot | \Z\in A):[0,1]^p \rightarrow [0,1]$ that is a copula (denoted as $C(\cdot | A)$ hereafter to be short).
Note that we have denoted inequalities componentwise. This will be our convention hereafter.

\mds

Now, Patton's seminal paper~\cite{patton2006a} has been referenced more than $2\,000$ times in the academic literature.
 The concept of conditional copulas (also sometimes called ``dynamic copulas'' or ``time-varying copulas'') has been applied in many fields: economics (\cite{patton2012},\cite{manner2012}), financial econometrics (\cite{jondeau2006},\cite{patton2009},\cite{christoffersen2012}), risk management (\cite{palaro2006},\cite{OhPtton2018}), agriculture (\cite{goodwin2015}), actuarial science (\cite{brechmann2013},\cite{fangmadsen2013} and~\cite{czado2019} more recently), hydrology (\cite{kim2016},\cite{hesami2016}), etc, among many others.
The rise of pair-copula constructions, particularly vine models (\cite{aas},\cite{bedfordcooke2001,bedfordcooke2002}) has fuelled the interest around conditional copulas. Indeed, generally speaking, any $p$-dimensional distribution can be described by $p(p-1)/2$ bivariate conditional copulas and $p$ margins. Even if most vine models assume that such conditional copulas are usual copulas (the so-called "simplifying assumption"; c.f.~\cite{hobaek2010,gijbels2016,DerumignyFermanian2017} and the references therein), there is here no consensus. Therefore, some recent papers propose some model specification for vines and the associated inference procedures by working directly on conditional copulas: see~\cite{schellhase2016},\cite{spanhel2017},\cite{kurzspanhel2017},\cite{SpanhelKurz2019}, e.g.

\mds

Moreover, the statistical theory of conditional copulas is currently an active research topic. In the literature, the conditioning subset $A$ in~(\ref{def_cond_cop}) is pointwise most often, i.e. the authors consider $A:= \{\Z=\z\}$ for some particular vector $\z\in \Rb^q$.
Typically, in a semi-parametric model, it is assumed that $C_{\X|\Z}(\x | \Z=\z)= C_{\theta(\z)}(\x)$ for some map $\z\mapsto \theta(\z)\in \Rb^m$ and the main goal is to statistically estimate the latter link function, as in~\cite{acar2011,acar2013,Abegaz,vatter2015}. Under a nonparametric point-of-view, the main quantity of interest is rather the empirical copula process given $(\Z=\z)$. For instance,~\cite{veraverbeke2011,gijbels2015,portiersegers2018} study the weak convergence of such a process.

\mds

To the best of our knowledge, almost all the papers in the literature until now have focused on pointwise conditioning events.
In a few papers, some box-type conditioning events as $A:=\prod_{k=1}^q \1\big(Z_k\in (a_k,b_k)\big)$ are considered, where $(a_k,b_k)\in \bar \Rb^2$ for every $k\in \{1,\ldots q\}$.
For example,~\cite{SchmidConditionalSpearman}, p.1127, discusses a Spearman's rho between two random variables $X_1$ and $X_2$, knowing that $X_1$ and/or $X_2$ is above (or below) some threshold. Nonetheless, the limiting law of such a quantity is not derived.
%because ``its analytics and the nonparametrical statistical inference are difficult''.
In the same spirit,~\cite{DuranteJaworski2010} estimate similar quantities for measuring contagions between two markets, but they do not yield their asymptotic variances.
They wrote that ``this variance is usually difficult to get in a closed form and can be estimated by means of a bootstrap procedure''. See~\cite{DurantePappada} too. \textcolor{black}{Indeed, the limiting law of such statistics cannot be easily deduced from the asymptotic behavior of the usual empirical copula process, and necessitate particular analysis (see below). The aim of our paper is to state general theoretical results to solve such problems.}

\mds
Actually, such box-type conditioning events provide a natural framework in many situations. 
\textcolor{black}{For instance, it is often of interest to measure and monitor conditional dependence measures between the components of $\X$ 
given $\Z$ belongs to some particular areas in $\Rb^q$, through a model-free approach.
Therefore, bank stress tests will focus on $A:=(Z_k > q^Z_k, k\in \{1,\ldots,q\})$ 
for some quantiles $q^Z_k$ of $Z_k$. 
Since the levels of the latter quantiles are often high, it is no longer possible to rely on marginal or joint estimators given pointwise conditioning events  (kernel smoothing, e.g.). This justifies the bucketing of $\Z$ values.      
Moreover, when dealing with high-dimensional vectors of covariates, discretizing the $\Z$-space is often the single feasible way of measuring conditional dependencies. Indeed, it is no longer possible to invoke usual nonparametric estimators, due to the usual curse of dimensionality.
Since dependence measures are functions of the underlying copula, the key theoretical object will be here the conditional copula $C(\cdot| A)$ of $\X$ given $(\Z\in A)$ for some borelian subsets $A$, and some of its nonparametric estimators.
%Again, this would yield a realistic and relevant way of monitoring conditional dependencies for high-dimension vectors $\Z$.
}

%\textcolor{black}{Moreover, when dealing with high-dimensional vectors of covariates, discretizing is often the single feasible way of building non- and %semiparametric dependence models in practice. 
%Indeed, it is no longer possible to invoke usual nonparametric estimators, due to the usual curse of dimension.
%Therefore, it is very common to rely on some reduction of dimension techniques. 
%In the case of conditional distributions/copulas, a simple (but rather rough) procedure would be to gather the covariates: assume there exists a partition
%$(A_1,\ldots,A_m)$ of $\Rb^q$, i.e. $\cup_{k=1}^m A_k=\Rb^q$, and $m$ copulas $C_{\X|A_k}$ on $[0,1]^p$ s.t.
%\begin{eqnarray}
%\lefteqn{ \Pb \big(\X\leq \x | \Z=\z \big)= \sum_{k=1}^m \1(\z\in A_k) \Pb \big(\X\leq \x | \Z\in A_k \big) } \nonumber \\
%&=&\sum_{k=1}^m \1(\z\in A_k) C_{\X|A_k}\big(\Pb(X_1\leq x_1 | \Z\in A_k),\ldots,\Pb(X_p\leq x_p | \Z\in A_k) \big).
%\label{model_copulas_boxes}
%\end{eqnarray}
%Obviously, the latter maps $C_{\X|A_k}$ are the conditional copulas given $(\Z\in A_k)$, $k\in \{ 1,\ldots,m\}$.
%}

\mds

%Note that the latter model assumption~(\ref{model_copulas_boxes}) does not mean that the pointwise conditioning copula $C_{\X|\Z}(\cdot | \Z=\z)$ is equal to %$C_{\X|A_k}(\cdot)$ when $\z\in A_k$. This is true only if $\Pb(X_j\leq x_j | \Z=\z)=\Pb(X_j\leq x_j | \Z\in A_k)$ for every $\z\in A_k$, every $j\in %\{1,\ldots,p\}$ and any real number $x_j$. C.f. Prop. 11 in~\cite{DerumignyFermanian2017}.

The goal of this paper is threefold.
First, in Section~\ref{weak_convergence_section}, we state the weak convergence of the empirical copula process indexed by borelian subsets under minimal assumptions, extending~\cite{segers} written for usual copulas.
Second, we prove the validity of the exchangeable bootstrap scheme for the latter process in Section~\ref{bootstrap_approx}. This provides an alternative to the usual nonparametric Efron's bootstrap (\cite{FermaRaduWegkamp}) and the multiplier bootstrap~\cite{remillard2009} for bootstrapping copula models. Third, Section~\ref{dependence_measures} introduces a family of general ``conditional'' dependence measures as mappings of the latter copulas. This family virtually includes and generalizes
all dependence measures that have been introduced until now. We apply our theoretical results to prove their asymptotic normality. We state our results with independent and identically variables, leaving aside the extensions to dependent data for further studies.
It is important to note that our results obviously include the particular case of no covariate/conditioning event. 
Therefore, we contribute to the literature on usual copulas as much as on conditional copulas.
Finally, Section~\ref{financial_returns} provides an empirical application of the latter tools to study 
conditional dependencies between stock returns.

\section{Weak convergence of empirical copula processes indexed by families of subsets}
\label{weak_convergence_section}
\subsection{Single conditioning subset}
\label{weak_convergence_section_single}

Let us consider a borelian subset $A\subset \Rb^{q}$ so that $p_{A}:= \Pb(\Z\in A)$ is positive.
Let $\big((\X_1,\Z_1),\ldots,(\X_n,\Z_n)\big)$ be an i.i.d. sample of realizations of $(\X,\Z)\in \Rb^{p+q}$.
The conditional copula of $\X$ given $(\Z\in A)$, that will simply be denoted by $C(\cdot | A)$, can be estimated by
$$ \hat C_n (\u|A):=\frac{1}{n \hat p_{A}} \sum_{i=1}^n \1 \big(F_{n,1}(X_{i,1} | A) \leq u_1,\ldots, F_{n,p}(X_{i,p} | A)\leq u_p, \Z_i\in A  \big),\; \text{where}$$
$$ F_{n,k}(t|A):= \frac{1}{n\hat p_{A}}\sum_{i=1}^n \1(X_{i,k} \leq t , \Z_i\in A    ),\; k=\{1,\ldots,p\},$$
$$ \hat p_{A} := n^{-1}\sum_{i=1}^n \1( \Z_i\in A)=:\frac{n_A}{n} \simeq p_A.$$
Note that $n_A$ is the size of the sub-sample of the observations $\X_i$ s.t. $\Z_i\in A$. It is a random integer in $\{0,1,\ldots,n\}$.
When $n_A=0$, simply set $\hat p_A=0$ and $F_{n,k}(\cdot|A)=0$ formally.

\mds

The associated copula process is denoted as $\hat\Cb_n(\cdot|A)$, i.e. $\hat\Cb_n(\u|A):=\sqrt{n}\big( \hat C_n(\u | A) - C(\u |A) \big)$ for any $\u\in [0,1]^p$.
Equivalently, one can define the empirical copula as
$$ \bar C_n (\u|A):=\frac{1}{n \hat p_{A}} \sum_{i=1}^n \1 \big(X_{i,1} \leq F_{n,1}^{-1}(u_1|A),\ldots,
X_{i,p}\leq F_{n,p}^{-1}(u_p|A), \Z_i\in A  \big),$$
invoking usual generalized inverse functions: $F^{-1}(u):=\inf \{t\in \Rb | F(t)\geq u\}$ for every univariate distribution $F$.
The associated copula process becomes $\bar\Cb_n(\cdot|A)$, where
$$\bar\Cb_n(\u|A):=\sqrt{n}\big( \bar C_n(\u | A) - C(\u |A) \big),\; \u\in [0,1]^p.$$

\mds

We assume hereafter that the conditional margins $F_k(\cdot | \Z\in A)$ are continuous, $k\in \{1,\ldots,p\}$.
First note that the asymptotic behaviors of $\hat\Cb_n(\cdot | A) $ and $\bar \Cb_n(\cdot | A)$ are the same.
Indeed, adapting the same arguments as in~\cite{RWZ}, Appendix C, it is easy to check that
$$ \sup_{\u\in [0,1]^p} | (\hat C_n- \bar C_n)(\u | A)| \leq \frac{p}{n_A \hat p_A},$$
almost everywhere, and then
\begin{equation}
 \sup_{\u\in [0,1]^p} | \sqrt{n}(\hat C_n - C)(\u | A) - \sqrt{n}( \bar C_n - C)(\u| A)|=o_P(1),
 \label{equiv_cop_processes}
 \end{equation}
since $\hat p_A$ tends to $p_A>0$ a.s. In other words, $ (\hat\Cb_n - \bar \Cb_n)(\cdot | A)$ tends to zero in probability in $\ell^\infty([0,1]^p)$, endowed with its sup-norm.
Therefore, the weak limits of $\hat\Cb_n(\cdot|A)$ and $\bar\Cb_n(\cdot|A)$ are the same.

\mds

In this section, we state the weak convergence of $\hat\Cb_n(\cdot |A)$ and/or $\bar\Cb_n(\cdot |A)$ in $\ell^\infty ([0,1]^p)$.
For convenience and w.l.o.g., we will focus on $\bar\Cb_n(\cdot | A)$ in the next theorem.

\mds

Second, the random variable $U^A_k:=F_k(X_k| \Z\in A)$ is uniformly distributed on $[0,1]$, given $(\Z\in A)$, for every $k\in \{1,\ldots,p\}$.
We denote by $\U^A$ the unobservable random vector $(U^A_1,\ldots,U^A_p)$, or simpler $\U$ when there is no ambiguity.
%~\footnote{Indeed, our results applies replacing $I$ by any subset $K\subset I$, and then replacing $p$ by the size of $K$.}.
For every $k\in \{1,\ldots,p\}$, the empirical distribution of the (unobservable) random variable $U^A_k$ given the event $(\Z \in A)$ is
$$ G_{n,k}(u |A)=n_A^{-1}\sum_{i=1}^n \1( U^A_{i,k} \leq u,\Z_i\in A), \; U^A_{i,k}:=F_k(X_{i,k}| \Z\in A), \; i\in \{1,\ldots,n\}.$$

Note that $G_{n,k}(u |A)$ and $F_{n,k}(t |A)$ can be seen as an average of $n_A$ indicator functions, i.e. an average on a sub-sample of observations whose size is random. Obviously, $G_{n,k}(u |A)$ tends to $\Pb( U_k^A \leq u | \Z\in A)=u$ a.e.
and its associated empirical process will be $\alpha_{n,k}(u|A):=  \sqrt{n_A}\big( G_{n,k}(u|A) - u\big)$, $u\in [0,1]$.
Note that the normalizing sample size is random here, contrary to the usual empirical processes. \textcolor{black}{Nonetheless, this will not be a source of worry for asymptotic behaviors and $n_A$ could be replaced by $n p_A$ in the definition of $\alpha_{n,k}(\cdot|A)$.}

\mds

Third, set
$$ \bar D_n(\u, A):= n^{-1} \sum_{i=1}^n \1 \big( U^A_{i,1}\leq G_{n,1}^{-1}(u_1|A),\ldots,U^A_{i,p}\leq G_{n,p}^{-1}(u_p|A), \Z_i\in A \big),$$
for any $\u\in [0,1]^p$, that tends to $D(\u, A):=\Pb( \U^A \leq \u, \Z\in A)=p_A\Pb( \U^A \leq \u | \Z\in A)$ a.s.
Note that $\big(X_{i,k}\leq F_{n,k}^{-1}(u|A)\big)$ if and only if $\big(U^A_{i,k}\leq G_{n,k}^{-1}(u |A)\big)$ for any $k\in \{1,\ldots,p\}$, $i\in \{1,\ldots,n\}$ and $u\in [0,1]$.
This implies
$$\bar C_n(\u|A)=\bar D_n(\u, A)/\hat p_A=\bar D_n(\u, A)/\bar D_n(\1, A) ,$$
and the asymptotic behavior of $\bar\Cb_n$ will be deduced from the weak convergence of the process $\bar \Db_n(\cdot,A)$, where
$\bar \Db_n(\u,A) := \sqrt{n}( \bar D_n - D)(\u,A)$.

\mds

The unfeasible empirical counterpart of $D(\u,A)$ is
$$ D_n(\u, A):= n^{-1} \sum_{i=1}^n \1 \big( \U^A_{i}\leq \u, \Z_i\in A \big).$$
A key process is $\Db_n(\cdot) := \sqrt{n}\big( D_n - D)(\cdot,A)$ that is a random map from $[0,1]^p$ to $\Rb$.
As every usual empirical process, it weakly tends in $\ell^\infty([0,1]^p)$ to a Brownian bridge.

\mds

In the meantime, define the instrumental empirical process
\begin{equation}
\widetilde \Db_n(\u,A):= \Db_n(\u,A) - p_A^{-1}\sum_{k=1}^p
\partial_k D\big(\u,A\big) \Big( \Db_n\big((u_k,\1_{-k}),A\big) - u_k \Db_n(\1,A)\Big),
\label{instrumental_process}
\end{equation}
denoting par $\partial_k D(\u,A) $ the partial derivative of the map $\u\mapsto D(\u,A)$ w.r.t. $u_k$.
This new process $\widetilde \Db_n(\cdot,A)$ will yield a nice approximation of the process of interest $\bar \Db_n(\cdot,A)$, as stated in the theorem below.

\mds

\begin{cond}
\label{cond_regularity_copula}
For every $k\in\{1,\ldots,p\}$, the partial derivative $\partial_k D(\u,A)$ of $D(\cdot,A)$ w.r.t. $u_k$ exists and is continuous on the set
$V_k:=\{\u\in [0,1]^p, 0<u_k<1\}$.
\end{cond}
The latter assumption is the standard ``minimal'' regularity condition, as stated in~\cite{segers}, so that the usual empirical copula process weakly converges in $\ell^\infty([0,1]^p)$.

\begin{thm}
\label{WeakConvCopProcess}
If $p_A>0$ and Condition~\ref{cond_regularity_copula} holds, then
$ \sup_{\u\in [0,1]^p} | (\bar\Db_n - \widetilde \Db_n)(\u,A) |$ tends to zero in probability.
\end{thm}
See the proof in the appendix, in Section~\ref{sec:proof:WeakConvCopProcess}. Note that $\widetilde \Db_n$ differs from the asymptotic approximation of the usual empirical copula process: compare $\widetilde\Db_n$ with Equation (3.2) and Proposition 3.1 in~\cite{segers}, for instance. This is due to the additional influence of the random sample size $n_A$, or, equivalently, the randomness of $\hat p_A$. This stresses that our results are not straightforward applications of the existing results in the literature.
%As a consequence of~(\ref{equiv_cop_processes}), the latter result is equivalent to $ \sup_{\u\in [0,1]^p} | (\hat\Db_n - \widetilde \Db_n)(\u,A)| =o_P(1)$.

\mds
Since the process $\Db_n$ is weakly convergent in $\ell^\infty([0,1]^p)$ - as any usual empirical process -, this yields the weak convergence of $\widetilde \Db_n$ and then of $\bar \Db_n$ in the same space.

\mds
\begin{cor}
If $p_A>0$ and Condition~\ref{cond_regularity_copula} holds, then the process $\bar\Db_n(\cdot,A)$ weakly converges in $\ell^\infty([0,1]^p)$ towards the centered Gaussian process
$ \Db_\infty(\cdot,A)$, where
$$ \Db_\infty(\u,A):= \Bb(\u,A) - p_A^{-1}\sum_{k=1}^p \partial_k D\big(\u,A\big) \Big( \Bb\big((u_k,\1_{-k}),A\big) - u_k \Bb(\1,A)\Big),$$
denoting by $\Bb(\cdot,A)$ a Brownian bridge, whose covariance function is given as
\begin{eqnarray*}
 \Eb\big[\Bb(\u,A) \Bb(\u',A) \big]&=& \Pb( \U^A \leq \u \wedge \u', \Z\in A) - \Pb( \U^A \leq \u , \Z\in A)\Pb( \U^A \leq \u' , \Z\in A) \\
 &=& p_A C_{\X|\Z}( \u \wedge \u' | \Z\in A) - p_A^2 C_{\X|\Z}( \u | \Z\in A)C_{\X|\Z}( \u' | \Z\in A) ,
\end{eqnarray*}
for every $(\u,\u')\in [0,1]^p$.
\label{cor_weak_convergence_Dn}
\end{cor}

\mds

Thus, we deduce the asymptotic behavior of
$ \bar C_n(\cdot | A)$ and $\bar \Cb_n$, since $ \bar C_n(\u | A) =  \bar D_n(\u, A)/ \bar D_n(\1 , A)$.
To this goal, recall that
$$C(\u |A)=\Pb(\U^A \leq \u | \Z\in A)=D(\u,A)/D(\1,A).$$
Therefore, simple algebra yield
\begin{eqnarray}
\lefteqn{ \bar\Cb_n(\u|A):=\sqrt{n}\big\{ \bar C_n(\u | A) - C(\u |A) \big\}=
\sqrt{n} \Big\{ \frac{\bar D_n(\u,A)}{\bar D_n(\1,A)} - \frac{D(\u,A)}{D(\1,A)}  \Big\}   \nonumber }\\
&=&  \sqrt{n} \bar D_n(\u,A) \Big\{ \frac{1}{\bar D_n(\1,A)} -  \frac{1}{ D(\1,A)}  \Big\}+
\frac{\sqrt{n}(\bar D_n- D)(\u,A)}{D(\1,A)}   \nonumber \\
&=&
 \bar D_n(\u,A)  \frac{ \sqrt{n}(D(\1,A) - \bar D_n(\1,A))}{  \bar D_n(\1,A) D(\1,A)}  +
\frac{\sqrt{n}(\bar D_n- D)(\u,A)}{D(\1,A)}   \nonumber \\
&=&
 \frac{\bar \Db_n(\u,A)}{p_A} - D(\u,A)  \frac{ \bar \Db_n(\1,A)}{  p_A^2 }  + o_P(1). \label{relation_Cbn_barDbn}
\end{eqnarray}

We deduce from the latter relationship and Corollary~\ref{cor_weak_convergence_Dn} that $\bar\Cb_n(\cdot|A)$ is weakly convergent in $\ell^\infty([0,1]^p)$.

\mds
\begin{thm}
\label{weak_conv_emp_copula_process_sets}
If $p_A >0$ and Condition~\ref{cond_regularity_copula} holds, then
$\hat\Cb_n$ and $\bar \Cb_n$ weakly tend to a centered Gaussian process $\Cb_\infty(\cdot|A)$ in $\ell^\infty([0,1]^p)$, where
\begin{eqnarray*}
\lefteqn{   \Cb_\infty(\u|A):= \frac{\Db_\infty(\u,A)}{p_A} - D(\u,A)  \frac{\Db_\infty(\1,A)}{  p_A^2 } }\\
&=& \frac{ \Bb(\u,A)}{p_A}
 - \sum_{k=1}^p \frac{\partial_k D\big(\u,A\big)}{p_A^{2}}  \Big( \Bb\big((u_k,\1_{-k}),A\big) - u_k \Bb(\1,A)\Big)
- \frac{D(\u,A)}{p_A^2} \Bb(\1,A).
\end{eqnarray*}
\end{thm}
By simple calculations, we explicitly write the covariance function of the limiting conditional copula process $\Cb_\infty(\u|A)$. Moreover, the latter 
covariance can be empirically estimated: see Appendix~\ref{covariances}.

\mds

When there is not conditioning subset, or when $A=\Rb^q$ equivalently, then $p_A=1$ and $\Bb(\1,A)=0$ a.s. (its variance is zero). In this case, we see that $\Cb_\infty(\u|A)$ becomes the well-known weak limit of the usual empirical copula process, as in~\cite{FermaRaduWegkamp,segers}.
\textcolor{black}{Nonetheless, we stress that Theorem~\ref{weak_conv_emp_copula_process_sets} cannot be straightforwardly deduced from the weak convergence of usual empirical copula processes, due to the dependencies between $\X$ and $\Z$.}

\mds

\begin{rem}
Theorem~\ref{weak_conv_emp_copula_process_sets} is not a consequence of Theorem 5 in~\cite{RWZ} either, where the authors state the weak convergence of the usual empirical copula process in $\ell^\infty(\Gc)$ for some set of functions $\Gc$ from $[0,1]^p$ to $\Rb$. Indeed, first, such functions are assumed to be right-continuous and of bounded variation in the sense of Hardy-Krause (see their Assumption F) when we consider general borelian subsets $A$.
Second and more importantly, it is not possible to recover our processes $\hat\Cb_n(\cdot|A)$ or $\bar\Cb_n(\cdot|A)$ of interest with some quantities as $\int g \,d\Cb_n$ for some particular function $g$ and a usual empirical copula process $\Cb_n$.
\end{rem}

\subsection{Multiple conditioning subsets}
\label{multiple_subsets}

We now consider a finite family of borelian subsets $A_j\subset \Rb^{q}$ such that w have $p_{A_j}:= \Pb(\Z\in A_j)>0$ for every $j\in \{1,\ldots,m\}$ and a given $m > 0$.
Set $ \Ac:=\{A_1,\ldots,A_m\}$.
The subsets in $\Ac$ may be disjoint or not.
By the same reasonings as above in a $m$-dimensional setting, we can easily prove the weak convergence of the process
$\vec \Cb_n(\cdot|\Ac)$ defined on $[0,1]^{mp}$ as
$$ \vec \Cb_n(\vec\u|\Ac):=\big(\bar \Cb_n(\u_1|A_1),\ldots,\bar \Cb_n(\u_m| A_m)\big),$$
for every $\u_j\in [0,1]^p$, $j\in \{1,\ldots,m\}$, where $\vec\u := (\u_1, \dots, \u_m)$.

\begin{thm}
\label{weak_conv_emp_copula_process_sets_vectorial}
If, for every $j\in\{1,\ldots,m\}$, $p_{A_j} > 0$ and Condition~\ref{cond_regularity_copula} holds for $A=A_j$, then
$\vec\Cb_n(\cdot|\Ac)$ weakly tends to a multivariate centered Gaussian process $\vec\Cb_\infty(\cdot|\Ac)$ in $\ell^\infty([0,1]^{mp},\Rb^m)$, where
$$   \vec\Cb_\infty(\vec\u|\Ac):= \big( \Cb_\infty(\u_1|A_1),\ldots,\Cb_\infty(\u_m|A_m) \big),\;\; \u_j\in [0,1]^p, \; j\in \{1,\ldots,m\}.$$
\end{thm}

The proof is straightforward and left to the reader.
The latter result is obviously true replacing $\bar \Cb_n$ by $\hat \Cb_n$. It will be useful for building and testing the relevance of some partitions $\Ac$ of the space of covariates, in the spirit of Pearson's chi-square test. Typically, this means testing the equality between the copulas $C_n(\cdot|A_j)$ and $C_n(\cdot|A_k)$ for several couples $(j,k) \in \{1,\dots, m\}^2$.

\mds

We can specify the covariance function of $ \vec\Cb_\infty(\vec\u|\Ac)$ and  $\vec\Cb_\infty(\vec\u'|\Ac)$, for any vectors $\vec\u$ and $\vec\u'$ in $[0,1]^{mp}$
by recalling that
\begin{eqnarray*}
\lefteqn{   \Cb_\infty(\u_j|A_j):=  \frac{ \Bb(\u_j,A_j)}{p_{A_j}}
 - \sum_{i=1}^p \frac{\partial_i D\big(\u_j,A_j\big)}{p_{A_j}^{2}}  \Big( \Bb\big((u_{j,i},\1_{-i}),A_j\big) - u_{j,i} \Bb(\1,A_j)\Big)  }\\
&-& \frac{D(\u_j,A_j)}{p_{A_j}^2} \Bb(\1,A_j), \hspace{6cm}
\end{eqnarray*}
where $\u_j=(u_{j,1},\ldots,u_{j,p})$, $j\in \{1,\ldots,m\}$
and by noting that
\begin{align}
    \Eb\big[\Bb(\u_{j},A_j) \Bb(\u_{k},A_{k}) \big]
    &= \Pb( \U^{A_j} \leq \u_j, \U^{A_{k}} \leq \u_k, \Z\in A_j\cap A_{k}) \nonumber \\
    &- \Pb( \U^{A_{j}} \leq \u_j , \Z\in A_j)
    \, \Pb( \U^{A_{k}} \leq \u_k , \Z\in A_{k}) ,
\label{general_cov_B}
\end{align}
for every $(j,k)\in \{1,\ldots,m\}^2$. Note we have not imposed that the subsets of $A_j$ are disjoint. Nonetheless, in the case of a partition (disjoint 
subsets $A_k$), calculations become significantly simpler because of the nullity of 
$\Pb( \U^{A_j} \leq \u_j, \U^{A_{k}} \leq \u_k, \Z\in A_j\cap A_{k})$.

\mds

Simple (but tedious) calculations yield the covariance function of the limiting vectorial conditional copula process 
$\vec\Cb_\infty(\vec\u|\Ac)$. Moreover, the latter 
covariance can be empirically estimated: see Appendix~\ref{covariances}.

\section{Bootstrap approximations}
\label{bootstrap_approx}

The limiting laws of the previous empirical processes $\hat \Cb_n$, $\bar \Cb_n$ (or even $\hat \Db_n$ and $\bar \Db_n$) are complex.
Therefore, it is difficult to evaluate the weak limits of some functionals of the latter processes, in particular the asymptotic variances of some
test statistics that may be built from them.
The usual answer to this problem is to rely on bootstrap schemes. In this section, we study the validity of some bootstrap schemes for our particular empirical copula processes.
We will prove the validity of the general exchangeable bootstrap for such processes, a result that has apparently never been stated in the literature even in the case of usual copulas, to the best of our knowledge.
Moreover, we extend the nonparametric bootstrap and the multiplier bootstrap techniques to the case of conditioning events that have a non-zero probability (the case of pointwise events is dealt in~\cite{omelka2013}).

\mds

\subsection{The exchangeable bootstrap}
\label{exchangeable_bootstrap}

For the sake of generality, we rely on the exchangeable bootstrap (also called ``wild bootstrap'' by some authors), as introduced in~\cite{VVW}.
For every $n$, let $\W_n:=(W_{n,1},\ldots,W_{n,n})$ be an exchangeable nonnegative random vector and $\bar W_n:=(W_{n,1}+\ldots,W_{n,n})/n$ its average.
For any borelian subset $A$, $p_A>0$, the weighted empirical bootstrap process of $\Db_n(\cdot,A)$ that is related to our initial i.i.d. sample $(\X_i,\Z_i)_{i=1,\ldots,n}$ is defined as
\begin{eqnarray*}
\lefteqn{ \Db^*_n (\u,A):= \frac{1}{\sqrt{n}} \bigg(\sum_{i=1}^n W_{n,i} \big\{\1( X_{i,1} \leq F_{n,1}^{-1}(u_1 |A),\ldots,X_{i,p} \leq F_{n,p}^{-1}(u_p |A),\Z_i\in A)
- \bar D_n(\u,A)\big\} \bigg)    }\\
&=& \frac{1}{\sqrt{n}}\sum_{i=1}^n W_{n,i} \1 \Big( X_{i,1} \leq F_{n,1}^{-1}(u_1 |A),\ldots,X_{i,p} \leq F_{n,p}^{-1}(u_p |A),\Z_i\in A \Big)
- \sqrt{n} \, \bar W_n \bar D_n(\u,A) .
\end{eqnarray*}
\mds
We require the standard conditions on the weights (see Theorem (3.6.13) in~\cite{VVW}).

\mds

\begin{cond}
\label{regul_weights_bootstrap}
$$ \sup_n \int_0^\infty \sqrt{\Pb\big( |W_{n,1}- \bar W_n| > t \big)}\, dt <\infty,$$
$$ n^{-1/2} \Eb \big[ \max_{1\leq i \leq n} | W_{n,i} - \bar W_n | \big] \stackrel{P}{\longrightarrow} 0,\;\text{and}\; n^{-1} \sum_{i=1}^n (W_{n,i} - \bar W_n)^2  \stackrel{P}{\longrightarrow} 1.$$
\end{cond}

Note that $\Db^*_n (\u,A)$ can be calculated, contrary to $\Db_n(\cdot,A)$. Since its asymptotic law will be ``close to'' the limiting law of $\Db_n(\cdot,A)$ when $n$ tends to the infinifty, resampling many times the vector $\W_n$ allows the calculation of many realizations of $\Db^*_n (\u,A)$, given the initial sample. This yields a numerical way of approximating the limiting law of $\Db_n (\u,A)$ or some functionals of the latter process. This is the usual and fruitful idea of most resampling techniques.

\mds

The same reasoning will apply to the copula processes $\hat \Cb_n(\cdot|A)$ and $\bar \Cb_n(\cdot|A)$, due to the relationships~(\ref{instrumental_process}) and~(\ref{relation_Cbn_barDbn}): to prove the validity of an exchangeable bootstrap scheme for the latter copula processes, we first approximate the unfeasible process $\Db_n(\cdot,A)$ by the weighted empirical bootstrapped process $\Db^*_n (\cdot,A)$; second, we invoke Theorem~\ref{WeakConvCopProcess} to obtain a similar results for $\bar \Db_n(\cdot,A)$; third, we use the relationship between $\bar\Db_n(\cdot,A)$ and $\bar\Cb_n(\cdot | A)$ and deduce a bootstrap approximation for our ``conditioned'' copula processes.

\mds

To be specific, let us consider $M$ independent realizations of the vector of weights $\W_n$ (that are mutually independent draws and independent of the initial sample), and the associated processes
$\Db^*_{n,k}(\cdot,A)$, $k\in \{1,\ldots,M\}$.
We first prove the validity of our bootstrap scheme for $\Db_n(\cdot,A)$.
Denote by $\Dc_{n,M,A}^*$ the process defined on $[0,1]^{p(M+1)}$ as
$$\Dc_{n,M,A}^*(\u_0,\u_1,\ldots,\u_M):=
\big(\Db_{n}(\u_0,A),\Db^*_{n,1}(\u_1,A),\ldots,\Db^*_{n,M}(\u_M,A) \big),$$
for every vectors $\u_0,\ldots,\u_M$ in $[0,1]^p$.
Moreover, denote by $\vec\Bb_\infty(\cdot,A)$ a process on $[0,1]^{p(M+1)}$ that concatenates $M+1$ independent versions of the Brownian bridge $\Bb(\cdot,A)$ introduced in Corollary~\ref{cor_weak_convergence_Dn}.

\mds

\begin{thm}
Under Condition~\ref{regul_weights_bootstrap}, for any $M\geq 2$ and when $n\rightarrow \infty$,
the process $\Dc_{n,M,A}^*$ weakly tends to $\vec\Bb_\infty(\cdot,A)$ in $\ell^{\infty}([0,1]^{p(M+1)},\Rb^{M+1})$.
\label{validity_boot_Dcn}
\end{thm}

See the proof in Section~\ref{proof:validity_boot_Dcn} of the appendix. 
The latter result validates the use of the considered bootstrap scheme. It has not to be confused with fidi weak convergence of $\Dc_{n,M,A}^*$, that is just a consequence of Theorem~\ref{validity_boot_Dcn}.

\mds

Thus, we can easily build a bootstrap estimator of $\widetilde \Db_n(\cdot,A)$, and then of $\bar \Db_n(\cdot,A)$.
Recalling Equation~(\ref{instrumental_process}), we evaluate the partial derivatives of $D(\cdot,A)$ as in~\cite{KojaSegersYan}: for every $\u\in [0,1]^p$,
\begin{equation}
 \partial_k D(\u,A) \simeq \widehat{\partial_k D}(\u,A):= \frac{1}{u_{k,n}^+ - u_{k,n}^-}
\Big( \bar D_n(\u_{-k},u_{k,n}^+,A) -\bar D_n(\u_{-k},u_{k,n}^-,A)  \Big),
\label{def_estimated_derivatives}
\end{equation}
where $u_{k,n}^+:=\min(u_k+n^{-1/2},1)$, $u_{k,n}^-:=\max(u_k-n^{-1/2},0)$ and with obvious notations.
Now, the bootstrapped version of $\widetilde \Db_n(\cdot,A)$ is defined as
\begin{equation}
\widetilde \Db^*_n(\u,A)
:= \Db^*_n(\u,A) - \hat p_A^{-1}\sum_{k=1}^p
\widehat{\partial_k D}\big(\u,A\big) \Big( \Db^*_n\big((u_k,\1_{-k}),A\big) - u_k \Db^*_n(\1,A)\Big).
\label{instrumental_process_boot}
\end{equation}
Importantly, note the latter process is a valid bootstrapped approximation of $\bar \Db_n(\cdot,A)$ too, because 
$\tilde \Db_n(\cdot,A)$ and $\bar \Db_n(\cdot,A)$ have the same limiting law (Theorem~\ref{WeakConvCopProcess}).

\mds
Denote by $\bar\Dc_{n,M,A}^*$ the process defined on $[0,1]^{p(M+1)}$ by
$$\bar\Dc_{n,M,A}^*(\u_0,\u_1,\ldots,\u_M):=
\big(\bar\Db_{n}(\u_0,A),\widetilde\Db^*_{n,1}(\u_1,A),\ldots,\widetilde \Db^*_{n,M}(\u_M,A) \big).$$
Moreover, denote by $\Dc_\infty(\cdot,A)$ a process on $[0,1]^{p(M+1)}$ that concatenates $M+1$ independent versions of $\Db_\infty(\cdot,A)$, as defined in Corollary~\ref{cor_weak_convergence_Dn}.
Then, we are able to state the validity of the exchangeable bootstrap for $\bar \Db_n$.

\begin{thm}
If $p_A>0$ and if Conditions~\ref{cond_regularity_copula} and~\ref{regul_weights_bootstrap} hold, then
the process $\bar\Dc_{n,M,A}^*$ weakly tends to $\Dc_\infty(\cdot,A)$ in $\ell^\infty([0,1]^{p(M+1)},\Rb^{M+1})$.
\end{thm}

\begin{proof}
\label{boot_barDbn_process}
With the same arguments as in the proof of Proposition 2 in~\cite{KojaSegersYan}, it can be proved that
$\sup_{\u\in [0,1]^p}|\widehat{\partial_k D}(\u,A)|\leq 5 $ for every $k\in \{1,\ldots,p\}$.
Moreover, by Lemma 2 in~\cite{KojaSegersYan}, for every $a,b$ s.t. $0<a<b<1$, we have
$$ \sup_{\u_{-k}\in [0,1]^{p-1}} \sup_{u_k\in [a,b]} |\partial_k D(\u,A)-\widehat{\partial_k D}(\u,A)| \stackrel{P}{\longrightarrow} 0. $$
By applying the same arguments as in Proposition 3.2 in~\cite{segers}, we obtain the result.
\end{proof}

Recalling Equation~(\ref{relation_Cbn_barDbn}), we deduce an exchangeable bootstrapped version of $\bar\Cb_n$, defined as
\begin{equation}
\widetilde \Cb^*_{n}(\u | A) := \frac{\widetilde \Db^*_n(\u,A)}{\hat p_A} - \bar D_n(\u,A)  \frac{ \widetilde \Db^*_n(\1,A)}{  \hat p_A^2 }\cdot
\label{boot_Cbn_process}
\end{equation}
Still considering $M$ independent random realizations of $\W_n$, we finally introduce the joint process
$\Cc_{n,M,A}^*$ whose trajectories are
$$(\u_0,\u_1,\ldots,\u_M)\mapsto \Cc_{n,M,A}^*(\u_0,\ldots,\u_M):=
\big(\bar\Cb_{n}(\u_0 | A),\widetilde\Cb^*_{n,1}(\u_1 | A),\ldots,\widetilde \Cb^*_{n,M}(\u_M | A) \big),$$
for every $\u_0,\ldots,\u_M$ in $[0,1]^p$.

\begin{cor}
\label{validity_bootstrap_barCn}
If $p_A>0$ and if Conditions~\ref{cond_regularity_copula} and~\ref{regul_weights_bootstrap} hold, then, for every $M\geq 2$ and when $n\rightarrow \infty$, the process
$\Cc_{n,M,A}^*$ weakly tends in $\ell^\infty([0,1]^{p(M+1)},\Rb^{M+1})$ to a process that
concatenates $M+1$ independent versions of $\Cb_\infty(\cdot | A)$, as defined in Theorem~\ref{weak_conv_emp_copula_process_sets}.
\end{cor}
In other words, we can approximate the limiting law of $\bar\Cb_{n}(\u | A)$
by the law of $\widetilde \Cb^*_{n}(\u | A)$, that is obtained by simulating many times independent realizations of the vector of weights $\W_n$, given the initial sample $(\X_i,\Z_i)_{i=1,\ldots,n}$.

\mds
\begin{rem}
Let $(\xi_{i})_{i\geq 1}$ be a sequence of i.i.d. random variables, with mean zero and variance one.
Formally, we can set $w_{n,k}=\xi_k$ for every $n$ and every $k\in \{1,\ldots,n\}$, even if the $\xi_i$ are not always nonnegative.
The same formulas as before yield some feasible bootstrapped processes that are similar to those obtained with the multiplier bootstrap of~\cite{remillard2009}, or in~\cite{segers}, Prop. 3.2. With the same techniques of proofs as above, it can be proved that this bootstrap scheme is valid, invoking Theorem 10.1 and Corollary 10.3 in~\cite{kosorok} instead of Theorem 3.6.13 in~\cite{VVW}. Therefore, we can state that Corollary~\ref{validity_bootstrap_barCn} applies, replacing $\W_n$ with i.i.d. normalized weights. In other words, the multiplier bootstrap methodology applies with empirical copula processes ``indexed by'' borelian subsets.
\end{rem}

\mds

It is straightforward to state some extensions of the latter results when considering several subsets $A_1,\ldots,A_m$ simultaneously, as in Section~\ref{multiple_subsets}.
With the same notations, let us do this task in the case of our previous bootstrap estimates.
To this goal, denote $\Ac=\{A_1,\ldots,A_m\}$, $\vec\u_j:=(\u_{j,1},\ldots,\u_{j,m})$, $\u_{j,k}\in [0,1]^p$ for every $j\in \{0,1,\ldots,M\}$, $k\in \{1,\ldots,m\}$,
$$ \vec \Cb^*_{n,j}(\vec\u_j| \Ac):=\big(\widetilde \Cb^*_n(\u_{j,1} | A_1),\ldots,\widetilde \Cb^*_n(\u_{j,m} | A_m)\big),\;\text{and}$$
$$ \vec\Cc^*_{n,M,\Ac}(\vec\u_0,\ldots,\vec\u_M):=
\big(\vec \Cb_n(\vec\u_0|\Ac),\vec\Cb^*_{n,1}(\vec\u_1|\Ac),\ldots, \vec\Cb^*_{n,M}(\vec\u_M|\Ac) \big).$$

\mds
\begin{thm}
If $p_{A_k}>0$ and Condition~\ref{cond_regularity_copula} is satisfied for every $A_k$, $k\in \{1,\ldots,m\}$ and if Condition~\ref{regul_weights_bootstrap} holds, then, for every $M\geq 2$ and when $n\rightarrow \infty$, the process
    $\vec\Cc^*_{n,M,\Ac}$ weakly converges in $\ell^\infty([0,1]^{mp(M+1)},\Rb^{m(M+1)})$ to a process that concatenates $M+1$ independent versions of $\vec\Cb_\infty(\cdot|\Ac)$ (as defined in Theorem~\ref{weak_conv_emp_copula_process_sets_vectorial}).
    \label{prop:exchangeable_bootstrap_copula}
\end{thm}

\subsection{The nonparametric bootstrap}

When $\W_n$ is drawn along a multinomial law with parameter $n$ and probabilities $(1/n,\ldots,1/n)$, we recover the original idea of Efron's usual nonparametric bootstrap, here applied to the estimation of the limiting law of $\Db_n(\cdot,A)$. Nonetheless, our final bootstrap counterparts $\widetilde\Cb^*_n(\cdot |A)$ for $\hat\Cb_n(\cdot | A)$ or $\bar \Cb_n(\cdot | A)$ are not the same as the commonly met nonparametric bootstrap processes. 
In particular, our methodology is analytically more demanding
than what is commonly met with nonparametric bootstrap schemes.
Indeed, the usual way of working in the latter case is simply to resample with replacement the initial sample and to recalculate the statistics of interest with the bootstrapped sample exactly in the same manner as with the initial sample. In practical terms, all analytics and IT codes can be reused as many times as necessary without any additional work.
This is not really the case when using the exchangeable bootstrap above, even in the simple case of multinomial weights:
the formulas~(\ref{instrumental_process_boot}) or~(\ref{boot_Cbn_process}) necessitate to ``rework'' the initial estimation procedures. In particular, it is necessary to write our statistics of interest $T_n$ as $T_n=\Phi\big(\Cb_n(\cdot | A)\big)+o_P(1)$ for some regular functional $\Phi$. Thus, the bootstrapped statistic is $T_n^*:=\Phi\big(\tilde \Cb^*_n(\cdot | A)\big)$. Sometimes, specifying $\Phi$ may be boring because of the use of multiple step estimators and/or nuisance parameters.

\mds

This additional stage (the calculation of $\Phi$) can be avoided. Indeed, note that the empirical copula $\bar C_n(\cdot |A)$ may be seen as a regular functional of $F_n$, the usual empirical distribution of $(\X_i,\Z_i)_{i=1,\ldots,n}$, i.e.
$\bar C_n = \psi(F_n)$. Now, it is tempting to apply Efron's initial idea by resampling with replacement $n$ realizations of $(\X,\Z)$ among the initial sample, and to set
$\bar C^*_n = \psi(F^*_n)$, $F_n^*$ being the empirical cdf associated to the bootstrapped sample $(\X_i^*,\Z_i^*)_{i=1,\ldots,n}$. Actually, this standard bootstrap scheme is valid but under slightly stronger conditions than for the exchangeable bootstrap schemes of Section~\ref{exchangeable_bootstrap}. In the case of the usual empirical copula process, the validity of this nonparametric bootstrap has been proven in~\cite{FermaRaduWegkamp} by applying the functional Delta-Method. Similarly, this technique can be applied in our case.

\mds

To be specific, for every $\x\in \Rb^p$, set
$$ F_{n}(\x|A):= \frac{1}{n\hat p_{A}}\sum_{i=1}^n \1(\X_{i} \leq  \x, \Z_i\in A    ),$$
the empirical counterpart of $F(\x|A)$.
Let $F_n$ be the empirical cdf of $(\X_i,\Z_i)_{i=1,\ldots,n}$.
Note that $F_n(\cdot|A)=\chi( F_n)(\cdot)$  for some functional $\chi$ from the space of cadlag functions on $\Rb^{p+q}$, with values in 
the space of cadlag functions on $\Rb^{p}$, and defined by
$$ \chi(F)(\x_0)=\int \1(\x\leq \x_0,\z\in A) \, F(d\x,d\z) \, / \int \1(\z\in A) \, F(d\x,d\z),\; \x_0\in \Rb^p.$$
It is easy to check that the latter function $\chi$ is Hadamard differentiable at every cdf $F$ on $\Rb^{p+q}$ s.t. $\int \1(\z\in A)\, F(d\x,d\z)>0$. Its derivative at $F$ is given by
\begin{eqnarray*}
\lefteqn{ \chi'(F)(h)(\x_0)=
\frac{ \int \1(\x\leq \x_0,\z\in A) \, h(d\x,d\z)}{ \int \1(\z\in A) \, F(d\x,d\z)}   }\\
&-& \Big(\int \1(\x\leq \x_0,\z\in A) \, F(d\x,d\z) \Big)
\frac{ \int \1(\z\in A) \, h(d\x,d\z)}{ \Big(\int \1(\z\in A) \, F(d\x,d\z) \Big)^2} \cdot
\end{eqnarray*}
Moreover, $\bar C_n(\cdot |A)=\phi\big(F_n(\cdot |A)\big)$, introducing a map $\phi$ from the space of cadlag functions on $\Rb^p$ to $\ell^\infty([0,1]^p)$ by
$$ \phi(F)(\u)=F\big( F_1^-(u_1),\ldots,F_p^-(u_p)   \big) .$$
Assume the copula $C(\cdot |A)$ is continuously differentiable on the whole hypercube $[0,1]^p$, a stronger assumption than our Condition~\ref{cond_regularity_copula}, as pointed out by~\cite{segers}. Then, Lemma 2 in~\cite{FermaRaduWegkamp} states that $\phi$ is Hadamard-differentiable tangentially to $C_0([0,1]^d)$, the space of continuous maps on $[0,1]^p$.
By the chain rule (Lemma 3.9.3 in~\cite{VVW}), this means that $\psi=\phi \circ \chi $ is still Hadamard differentiable tangentially to $C_0([0,1]^d)$ and its derivative is
$\psi'(F)=\phi' \big( \chi (F) \big) \circ \chi'(F).$ This is the main condition to apply the Delta-Method for bootstrap (Theorem 3.9.11 in~\cite{VVW}, e.g.).

\mds
The nonparametric bootstrapped empirical copula associated with $\bar C_n(\cdot|A)$ is then defined as
$$ \bar C^*_n (\u|A):=\frac{1}{n \hat p^*_{A}} \sum_{i=1}^n \1 \big(X^*_{i,1} \leq (F^*_{n,1})^{-1}(u_1|A),\ldots,
X^*_{i,p}\leq (F^*_{n,p})^{-1}(u_p|A), \Z^*_{i}\in A  \big),$$
and the associated bootstrapped copula process is given by
$$\bar\Cb^*_n(\u|A):=\sqrt{n}\big( \bar C^*_n(\u | A) - \bar C_n(\u |A) \big),\; \u\in [0,1]^p.$$
Obviously, $F_n^*$ (resp. $F_{n,k}^*$) is the associated empirical cdf (resp. empirical marginal cdfs') associated to the nonparametric bootstrap sample $(\X_i^*,\Z_i^*)_{i=1,\ldots,n}$.
By mimicking the arguments of~\cite{FermaRaduWegkamp}, Theorem 5, it is easy to state the validity of the nonparametric bootstrap scheme for $\bar C_n(\cdot|A)$.
Details are left to the reader.
To simply announce the result, introduce the random map
$$ \Cc_{n,M,A}(\u_0,\ldots,\u_M):=
\big(\bar \Cb_n(\u_0|A),\bar\Cb^*_{n,1}(\u_1|A),\ldots, \bar\Cb^*_{n,M}(\u_M|A) \big),$$
for every vectors $\u_0,\ldots,\u_M$ in $[0,1]^p$.

\mds

\begin{thm}
    If the copula $C(\cdot |A)$ is continuously differentiable on $[0,1]^p$ and $p_A>0$, then the process
    $\Cc_{n,M,A}$ weakly converges in $\ell^\infty([0,1]^{p(M+1)},\Rb^{m(M+1)})$ to a process that
    concatenates $M+1$ independent versions of $\Cb_\infty(\cdot | A)$.
\end{thm}

\mds

As for the exchangeable bootstrap case, we can extend the latter results when dealing with several subsets $A_1,\ldots,A_m$ simultaneously.
Then,
still considering $m$ borelian subsets in $\Ac=\{A_1,\ldots,A_m\}$, for every
$\vec\u_j:=(\u_{j,1},\ldots,\u_{j,m})$, $\u_{j,k}\in [0,1]^p$ for every $j\in \{0,1,\ldots,M\}$, $k\in \{1,\ldots,m\}$, set
$$ \vec \Eb^*_{n,j}(\vec\u_j,\Ac):=\big(\bar \Cb^*_n(\u_{j,1},A_1),\ldots,\bar \Cb^*_n(\u_{j,m},A_m)\big),\;\text{and}$$
$$ \vec\Ec_{n,M,\Ac}(\vec\u_0,\ldots,\vec\u_M):=
\big(\vec \Cb_n(\vec\u_0|\Ac),\vec\Eb^*_{n,1}(\vec\u_1|\Ac),\ldots, \vec\Eb^*_{n,M}(\vec\u_M|\Ac) \big).$$

\mds
\begin{thm}
If the copulas $C(\cdot |A_k)$ are continuously differentiable on $[0,1]^p$ and $p_{A_k}>0$ for every $k\in\{1,\ldots,m\}$, then,
for every $M\geq 2$ and when $n\rightarrow \infty$, the process
    $\vec\Ec_{n,M,\Ac}$ weakly converges in $\ell^\infty([0,1]^{p(M+1)m},\Rb^{M+1})$ to a process that concatenates $M+1$ independent versions of $\vec\Cb_\infty(\cdot |\Ac)$.
    \label{prop:nonparametric_bootstrap_copula}
\end{thm}

\mds

\section{Application to Generalized dependence measures}
\label{dependence_measures}

\subsection{A single conditioning subset}

Dependence measures (also called ``measures of concordance'' or ``measures of association'' by some authors; see~\cite{Nelsen1999}, Def. 5.1.7.) are real numbers that summarize the amount of dependencies across the components of a random vector. Most of the time, they are defined for bivariate vectors, as originally formalized in~\cite{scarsini1984}. The most usual ones are Kendall's tau, Spearman's rho, Gini's measures of association and Blomqvist's beta.
Denoting by $C$ the copula of a bivariate random vector $(X_1,X_2)$, all these measures can be rewritten as weighted sums of quantities as
$\rho_1(\psi,\alpha):=\int \psi(u,v)\, C^\alpha(u,v ) C(du,dv )$ for some measurable map $\psi:[0,1]^2\rightarrow \Rb$, $\alpha\geq 0$,
or as $\rho_2(\psi,\alpha,\mu):=\int \psi(u,v)\, C^\alpha(u,v ) \mu(du,dv )$ for some measure $\mu$ on $[0,1]^2$.
For example, in the case of Kendall's tau (resp. Spearman's rho), the first case (resp. second case) applies by setting $\psi=1$ and $\alpha=1$ (resp. $\alpha=1$, $\mu(du,dv)=du\, dv$).
Gini's index is $\rho_1(\psi_G,0)$, with $\psi_G(u,v):=2\big(|u+v-1| - |u-v| \big)$.
Blomqvist's beta is obtained with $\rho_2(1,1,\delta_{(1/2,1/2)})$, where $\delta_{(1/2,1/2)}$ denotes the Dirac measure at $(1/2,1/2)$.
See~\cite{Nelsen1999}, Chapter 5, or~\cite{Nelsen2002} for some justifications of the latter results and additional results.

\mds

A few multivariate extensions of the latter measures have been introduced in the literature for many years.
The axiomatic justification of such measures for $p$-dimensional random vectors has been developed in~\cite{Taylor}, and many proposals followed, sometimes in passing. The most extensive analysis has been led in a series of papers by F. Schmid, R. Schmidt and some co-authors: c.f.~\cite{SchmidMultivariateSpearman,SchmidBlomqvist,SchmidConditionalSpearman,SchmmidEtAl2010}.

\mds

Actually, we can even more extend the previous ideas by considering general formulas for multivariate dependence measures,
possibly indexed by subsets (of covariates), as in the previous sections.
To be specific, we still consider a random vector $(\X,\Z)\in \Rb^p\times \Rb^q$ and we will be interested in dependence measures between the components of $\X$, when $\Z$ belongs to some borelian subset $A$ in $\Rb^{q}$.
For any (possibly empty) subsets $K$ and $K'$ that are included in $I:=\{1,\ldots,p\}$, let us define
\begin{equation}
\rho_{K,K'}(A)  :=\int \psi(\u)\, C_{K}(\u_K | \Z\in A ) C_{K'}(d\u_{K'} | \Z\in A )\, d\u_{I\setminus K'},
\label{defTau_KK_cond}
\end{equation}
for some measurable function $\psi$. Obviously, $C_{K}(\cdot | \Z\in A )$ denotes the conditional copula of $\X_K:=(X_j, j\in K)$ given $(\Z\in A)$.
In particular, $C_{I}(\u| \Z\in A )=C_{\{1,\ldots,p\}}(\u| \Z\in A )=C_{\X|\Z}(\u | \Z\in A )$, for every $\u\in [0,1]^p$.
When $K'=\emptyset$ (resp. $K'=I$) there is no integration w.r.t. $C_{K'}(d\u_{K'} | \Z\in A )$ (resp. $d\u_{I\setminus K'}$). 

\mds

The latter definition virtually includes and/or extends all unconditional and conditional dependence measures that have been introduced until now.
Indeed, such dependence measures are linear combinations (or even ratios, possibly) of our quantities $\rho_{K,K'}(A)$, for conveniently chosen $(K,K')$ and $\psi$.
Note that, by setting $A=\Rb^q$, we recover unconditional dependence measures. Moreover, setting $A=(\Z=\z)$ allows to study pointwise conditional dependence measures.

\mds

A few examples of such $\rho_{K,K'}(A)$ that have already been met in the literature:
\begin{itemize}
\item $\psi(\u)=1$, $K=K'=I$ and $A=\Rb^q$ provides a multivariate version of the Kendall's taus' of $\X$, that are affine functions of $\int C_{\X}(\u)\, C_{\X}(d\u)$.
See~\cite{Joe1990,GNBG,fermanian2015}, among others;
\item $\psi(\u)=1$, $K=I$, $K'=\emptyset$ and $A=\Rb^q$ yields $\rho_1$, the multivariate Spearmans's rho of $\X$, as in~\cite{SchmidMultivariateSpearman}; see~\cite{wolff1980} too.
\item $\psi(\u)=1$, $K=\emptyset$, $K'=I$ and $A=\Rb^q$ yields the multivariate Spearmans's rho of $\X$ introduced in~\cite{ruymgaart1978}, also called $\rho_2$ in~\cite{SchmidMultivariateSpearman};
\item $\psi(\u)=1$, $K=K'=I$ and a (small) neighborhood of $\z$ as $A$ is similar to a $p$-dimensional extension of the pointwise conditional Kendall's tau studied in~\cite{veraverbeke2011} or~\cite{DerumignyFermanian2019,DerumignyFermanian2020};
\item $\psi(\u)=\prod_{j\in I} 1(u_j\leq 1/2)$, $K=\emptyset$ and $K'=I$ corresponds to a conditional version of Blomqvist coefficient (\cite{Nelsen1999});
\item $\psi(\u)= 1(\u\leq \u_0 )+  1(\u\geq \v_0 )$, $K=\emptyset$ and $K'=I$ yields a conditional version of the tail-dependence coefficient considered in~\cite{SchmidBlomqvist};
\item if $\psi$ is a density on $[0,1]^p$, $K=I$ and $K'=\emptyset$, we get some conditional product measures of concordance, as defined in~\cite{Taylor};
\item when $\psi(\u)$ is a weighted sum of reflection indicators of the type
$$\u \in [0,1]^p \mapsto (\epsilon_1 u_1 + (1-\epsilon_1)(1-u_1),\ldots,(\epsilon_p u_p + (1-\epsilon_p)(1-u_p)\big),$$
where $\epsilon_k\in \{0,1\}$ for every $k\in \{1,\ldots,p\}$,
we obtain some generalizations of dependence measures (Kendall's tau, Blomqvist coefficient, etc), as introduced in~\cite{Joe1990}.
For conveniently chosen weights, such linear combinations of $\rho_{K,K'}(\Rb^q)$ for different subsets $K$ and $K'$ yield dependence measures that are increasing w.r.t. a so-called ``concordance ordering'' property. See~\cite{Taylor}, Examples 7 and 8, too. Etc.
\end{itemize}
Note that our methodology includes as particular cases some multivariate dependence measures that are calculated as averages of ``usual'' dependence measures when they are calculated for many pairs $(X_k,X_l)$, $k,l\in \{1,\ldots,p\}^2$. This old and simple idea (see~\cite{KendallBabington}) has been promoted by some authors.
See such type of multivariate dependence measures in~\cite{SchmmidEtAl2010} and the references therein.

\mds

Generally speaking, it is possible to estimate the latter quantities $\rho_{K,K'}(A)$ after replacing the conditional copulas by their estimates in Equation~(\ref{defTau_KK_cond}).
This yields the estimator
\begin{equation}
\hat\rho_{K,K'}(A)  :=
\int \psi(\u)\, \hat C_{n,K}(\u_K | \Z\in A ) \hat C_{n,K'}(d\u_{K'} | \Z\in A )\, d\u_{I\setminus K'},
\label{defTau_KK_cond_estimated}
\end{equation}
where we define
$$ \hat C_{n,K} (\u_K|A):=\frac{1}{n \hat p_{A}} \sum_{i=1}^n
\1 \big(F_{n,j}(X_{i,j} | A) \leq u_j, \forall j\in K, \Z_i\in A  \big),$$
and similarly for the induced measure $\hat C_{n,K'}(d\u_{K'} | \Z\in A )$.

\mds

Then, the weak convergence of the process $\hat \Cb_n(\cdot|A)=\sqrt{n}(\hat C_n - C)(\cdot|A)$ will provide the
limiting law of $\sqrt{n}\big(\hat \rho_{K,K'}(A) - \rho_{K,K'}(A)\big) $.
Indeed, the map
\begin{equation}
 \Psi_{K,K'}: C\mapsto  \int \psi(\u)\, C_{K}(\u_K ) C_{K'}(d\u_{K'} )\, d\u_{I\setminus K'}
 \label{def_map_Psi}
 \end{equation}
is Hadamard differentiable from $\Cc_p$, the space of cdfs' on $[0,1]^p$, onto $\Rb$.
To prove the latter result, for every $\u_{K'}\in [0,1]^{|K'|}$, denote
$$ \chi(\u_{K'}):=\int \psi(\u)\, C_{K}(\u_K ) \, d\u_{I\setminus K'}.$$

\begin{lemma}
\label{Hadamard_diff_rho}
If $\psi$ is continuous on $[0,1]^p$ and the map $\chi$ is of bounded variation on $[0,1]^{|K'|}$, then the map $\Psi_{K,K'} : \Cc_p \longrightarrow \Rb$ is Hadamard-differentiable at every $p$-dimensional copula $C$, tangentially to the set of real functions that are continuous on $[0,1]^p$.
Its derivative is given by
$$ \Psi_{K,K'}'(C)(h)= \int \psi(\u)\, h_{K}(\u_K ) C_{K'}(d\u_{K'} )\, d\u_{I\setminus K'} + \int \psi(\u)\, C_{K}(\u_K ) h_{K'}(d\u_{K'} )\, d\u_{I\setminus K'},$$
for any continuous map $h:[0,1]^p\rightarrow \Rb$.
\end{lemma}
When $h$ is not of bounded variation, we define the second integral of $\Psi_{K,K'}'(C)(h)$ by an integration by parts, as detailed in~\cite{RWZ}.
See the proof of Lemma~\ref{Hadamard_diff_rho} in the appendix, Section~\ref{sec:proof:Hadamard_diff_rho}.

\mds

As a consequence, by applying the Delta Method (Theorem 3.9.4 in~\cite{VVW}) to the copula process 
$\sqrt{n}\big(\hat C_n(\cdot | A)-  C(\cdot|A)\big)$, we obtain the asymptotic normality of $\hat\rho_{K,K'}(A)$.

\begin{thm}
\label{AN_generalized_dependence_measures}
Under the assumptions of Theorem~\ref{weak_conv_emp_copula_process_sets} and Lemma~\ref{Hadamard_diff_rho},
$$ \sqrt{n}\big( \hat\rho_{K,K'}(A) -  \rho_{K,K'}(A) \big) \weakly \Nc\big(0, \sigma^2_{K,K'}(A)   \big),$$
\begin{eqnarray*}
\lefteqn{  \sigma^2_{K,K'}(A)  := \var \Big(
\int \psi(\u)\, \Cb_{\infty,K}(\u_K |A ) C_{K'}(d\u_{K'} |A )\, d\u_{I\setminus K'} }\\
&+& \int \psi(\u)\, C_{K}(\u_K | A ) \Cb_{\infty,K'}(d\u_{K'} |A )\, d\u_{I\setminus K'}\Big).
\end{eqnarray*}
\end{thm}

\mds

As an example, let us consider the multivariate Spearman's rho obtained when setting
$\psi(\u)=1$, $K=I$, $K'=\emptyset$, $p=q$, $\X=\Z$ and $A=\prod_{j=1}^p ]-\infty,a_j]$, for some threshold $(a_1,\ldots,a_p)$ in $\Rb^p$. In other words, we focus on
$$ \rho_S(\a):=\int C_{\X}(\u |  X_j \leq a_j, \forall j\in \{1,\ldots,p\} )  \, \prod_{j=1}^p du_j .$$
This measure is related to the average dependencies among the components of $\X$, knowing that all such components are observed in their own tails. Indeed, we are interested in the joint tail $X_j \leq a_j$ for every $j\in\{1,\ldots,p\}$. Such an indicator has been introduced in~\cite{SchmidMultivariateSpearman} but its properties have not been studied. Indeed, the authors wrote: ``Certainly, this version would be interesting to investigate, too,
although its analytics and the nonparametrical statistical inference are difficult''. Therefore, they prefer to concentrate on other Spearman's rho-type dependence measures. Now, we fill this gap by applying Theorem~\ref{AN_generalized_dependence_measures}.
With our notations, a natural estimator of $\rho_S(\a)$ is
$$ \hat\rho_S(\a):=\int \hat C_{n}(\u |  X_j \leq a_j, \forall j\in \{1,\ldots,p\} )  \, \prod_{j=1}^p du_j .$$

\begin{cor}
If $p_A >0$ and Condition~\ref{cond_regularity_copula} holds, then
$$ \sqrt{n}\big( \hat\rho_{S}(\a) -  \rho_{S}(\a) \big) \weakly \Nc\big(0, \sigma^2_{S}(\a)   \big),\;\sigma^2_{S}(\a)  := \int \Eb\big[ \Cb_\infty(\u_1|A)  \Cb_\infty(\u_2|A)\big] \,d\u_1 \, d\u_2.$$
\end{cor}
The analytic formula of $\Eb\big[ \Cb_\infty(\u_1|A)  \Cb_\infty(\u_2|A)\big]$ is provided in Appendix~\ref{covariances}.
The asymptotic variance $\sigma^2_{S}(\a)$ can be consistently estimated after replacing the unknown quantities 
$C(\cdot |A)$, $p_A$, $D(\cdot,A)$ and its partial derivatives by 
some empirical counterparts, as in the latter appendix. 
Alternatively, the limiting law of $\sqrt{n}\big( \hat\rho_{S}(\a) -  \rho_{S}(\a) \big)$ can be obtained by several bootstrap schemes, as explained in Section~\ref{bootstrap_approx}. Indeed, since $ \sqrt{n}\big( \hat\rho_{S}(\a) -  \rho_{S}(\a) \big) = \int \hat\Cb_n(\u|A)\, d\u$, a bootstrap equivalent of the latter statistics is
$\int \tilde\Cb^*_n(\u,A)\, d\u$ or $ \int \bar\Cb^*_n(\u,A)\, d\u$, with the same notations as above and conveniently chosen bootstrap weights.

\mds

\subsection{Multiple conditioning subsets}

Important practical questions arise considering several borelian subsets simultaneously. For instance, is the amount of dependencies among the $\X$'s components the same 
when $\Z$ belongs to different subsets? This questioning can lead to a way of building relevant subsets $A_j$, $j\in \{1,\ldots,p\}$. 
Typically, a nice partition of the $\Z$-space is obtained when the copulas $C(\dot | \Z\in A_j)$ are heterogeneous. 
This is why we now extend the previous framework to be able to answer such questions.

\mds
To this goal, let $\Ac:=\{A_1,\ldots,A_m\}$ be a family of borelian subsets, $p_{A_j}>0$ for every $j\in \{1,\ldots,m\}$. Moreover, denote 
by $K_j,K'_j$, $j\in \{1,\ldots,m\}$ some subsets of indices in $I=\{1,\ldots,p\}$.  
To lighten notations, set 
$$\rho_j:=\int \psi_j(\u)\, C_{K_j}(\u_{K_j} | \Z\in A_j ) C_{K_j'}(d\u_{K_j'} | \Z\in A_j )\, d\u_{I\setminus K_j'} ,\;\text{and}$$
$$ \hat\rho_j  := \int \psi_j(\u)\, \hat C_{n,K_j}(\u_{K_j} | \Z\in A_j ) \hat C_{n,K_j'}(d\u_{K_j'} | \Z\in A_j )\, d\u_{I\setminus K_j'},$$
 for every $j$. Note that we allow different measurable maps $\psi_j$.
 
\mds

As above, we can deduce the asymptotic law of
$ \sqrt{n}\big(\hat \rho_1- \rho_1, \ldots,\hat \rho_m- \rho_m\big) ,$
from the weak convergence of the random vectorial process $\vec\u\mapsto \vec \Cb_n(\vec\u,\Ac)$ (Theorem~\ref{weak_conv_emp_copula_process_sets_vectorial}).
Denote by $\vec\Psi$ the map from $\Cc_p^m$ to $\Rb^m$ defined by
$$\vec \Psi (C_1,\ldots,C_m)=\big(\Psi_1(C_1),\ldots,\Psi_m(C_m)\big),$$
\begin{equation}
 \Psi_j: C\mapsto  \int \psi_j(\u)\, C_{K_j}(\u_{K_j} ) C_{K_j'}(d\u_{K_j'} )\, d\u_{I\setminus K_j'}.
 \label{def_map_Psi_extended}
 \end{equation}
Moreover, set $ \chi_j(C,\u_{K_j'}):=\int \psi_j(\u)\, C_{K_j}(\u_{K_j} ) \, d\u_{I\setminus K_j'}$ for every cdf $C$ on $[0,1]^p$. 
The next lemma is a straightforward extension of Lemma~\ref{Hadamard_diff_rho}.
Denote $\vec C:=(C_1,\ldots,C_m)$, for a given set of $m$ copulas $C_j$ on $[0,1]^p$.

\mds
\begin{lemma}
\label{Hadamard_diff_rho_extended}
If, for every $j\in \{1,\ldots,m\}$, the map $\psi_j$ is continuous on $[0,1]^p$ and $\chi_j(C_j,\cdot)$ is of bounded variation on $[0,1]^{|K_j'|}$, 
then $\vec\Psi$ is Hadamard-differentiable at $\vec C$, tangentially 
to the set of real functions that are continuous on $[0,1]^{mp}$.
Its derivative is given by
$$ \vec\Psi'(\vec C)(\vec h)=\Big(\Psi_{K_1,K_1'}'(C_1)(h_1),\ldots,\Psi_{K_m,K_m'}'(C_m)(h_m) \Big),$$
for any continuous map $\vec h:=(h_1,\ldots,h_m)$, $h_j:[0,1]^p\rightarrow \Rb$ for every $j$.
\end{lemma}

In the latter result, we have implicitly assumed that $\Psi_{K_j,K_j'}$ involves the $\psi_j$ function.
By the Delta method, we deduce the joint asymptotic normality of our statistics of interest.

\begin{thm}
\label{AN_generalized_extended}
If $p_{A_j} >0$ and Condition~\ref{cond_regularity_copula} holds with $A=A_j$, for every $j\in\{1,\ldots,m\}$, and 
under the assumptions of Lemma~\ref{Hadamard_diff_rho_extended}, then
$$ \sqrt{n}\big( \hat\rho_1 -  \rho_1,\ldots, \hat\rho_m -  \rho_m \big) \weakly \Nc\big(\0_m,\Sigma  \big),$$
where the components of the $m\times m$ matrix $\Sigma=[\Sigma_{k,l}]_{1\leq k,l \leq m}$ are
\begin{eqnarray*}
\lefteqn{  \Sigma_{k,l} := \int \psi_j(\u) \psi_k(\v) 
\Eb\Big[ 
\big\{  \Cb_{\infty,K_j}(\u_{K_j} |A_j ) C_{K_j'}(d\u_{K_j'} |A_j ) +  C_{K_j}(\u_{K_j} | A_j ) \Cb_{\infty,K_j'}(d\u_{K_j'} |A_j ) \big\}  }\\
&&\times  \big\{  \Cb_{\infty,K_k}(\v_{K_k} |A_k ) C_{K_k'}(d\v_{K_k'} |A_k ) +  C_{K_k}(\v_{K_k} | A_k ) \Cb_{\infty,K_k'}(d\v_{K_k'} |A_k ) \big\}
  \Big] \, d\u_{I\setminus K_j'}\, d\v_{I\setminus K_k'}.
\end{eqnarray*}
\end{thm}

As an application, let us consider the test of the zero assumption
$$ \Hc_0: C(\cdot |A) \;\text{does not depend on } A \in \Ac,\; \text{or equivalently}$$ 
$$ \Hc_0: C(\u|A_1)=\cdots =C(\u|A_m)  \;\text{for every } \u\in [0,1]^p.$$
This can be tackled through any generalized dependence measure $\hat\rho_{K,K'}(A)$, 
for some fixed subsets $K$ and $K'$ and a unique function $\psi$. 
In other words, with our previous notations, $\rho_j=\rho_{K,K'}(A_j)$ for every $j$.
Indeed, we can build a test statistic of the form
\begin{equation*}
    \Tc := \| (i,j) \mapsto \sqrt{n}\big(\hat\rho_{K,K'}(A_i) - \hat\rho_{K,K'}(A_j) \big)  \|,
\end{equation*}
where $\|$ is any semi-norm on $\Rb^{m^2}$.
For example, define the Cramer-von Mises type statistic
\begin{equation*}
    \Tc_{CvM} := n \sum_{j=2}^m
    \big( \hat\rho_{K,K'}(A_1) - \hat\rho_{K,K'}(A_j) \big)^2,
\end{equation*}
or the Kolmogorov-Smirnov type test statistic
\begin{equation*}
    \Tc_{KS} := \sqrt{n} \max_{j=2, \dots, m}
    \big| \hat\rho_{K,K'}(A_1) - \hat\rho_{K,K'}(A_j) \big|.
\end{equation*}
Note that under the null hypothesis, these test statistics can be rewritten as
\begin{align*}
    \Tc &= \| (i,j) \mapsto \sqrt{n} \big\{
    \hat\rho_{K,K'}(A_i) - \rho_{K,K'}(A_i) + \rho_{K,K'}(A_j) - \hat\rho_{K,K'}(A_j) \big\} \| \nonumber \\
    &= \| (i,j) \mapsto  \sqrt{n} \big\{
    \big(\hat\rho_{K,K'}(A_i) - \rho_{K,K'}(A_i) \big) -
    \big( \hat \rho_{K,K'}(A_j) - \rho_{K,K'}(A_j) \big) \big\}\|
\end{align*}
Therefore, under $\Hc_0$, Theorem~\ref{AN_generalized_extended} tells us that $\Tc$ (once properly rescaled) is weakly convergent. 
Since its limiting law is complex, we advise to use bootstrap approximations to evaluate the asymptotic p-values associate to $\Tc$ in practice, or simply its asymptotic variance.
A bootstrapped version of such tests statistics is
\begin{align*}
    \Tc^* &:= \| (i,j) \mapsto  \sqrt{n} \big\{ \hat\rho_{K,K'}^*(A_i) - \hat\rho_{K,K'}(A_i)
    + \hat\rho_{K,K'}(A_j) - \hat\rho^*_{K,K'}(A_j) \big\} \| ,
    %\Tc^{**} &:= \big|\big| (i,j) \mapsto \hat\rho_{K,K'}^*(A_i) - \hat\rho_{K,K'}^*(A_j) \big|\big|,
\end{align*}
where, in the case of the multiplier bootstrap, we set
$$    \hat\rho^*_{K,K'}(A)  :=
    \int \psi(\u)\, \widetilde C_{n,K}^*(\u_K | \Z\in A )
    \widetilde C_{n,K'}^*(d\u_{K'} | \Z\in A )\, d\u_{I\setminus K'},$$
and, in the case of the nonparametric bootstrap,
$$    \hat\rho^*_{K,K'}(A)  :=
    \int \psi(\u)\, \bar C_{n,K}^*(\u_K | \Z\in A )
    \bar C_{n,K'}^*(d\u_{K'} | \Z\in A )\, d\u_{I\setminus K'}.$$

%By combining the Delta Method (Theorem 3.9.4 in~\cite{VVW}) with Proposition~\ref{prop:nonparametric_bootstrap_copula}, we justify the validity of such bootstrap versions:
Under the assumptions of Theorem~\ref{prop:exchangeable_bootstrap_copula} (resp. Theorem~\ref{prop:nonparametric_bootstrap_copula}) and those of Theorem~\ref{AN_generalized_extended}, the couple
$ \big( \Tc_{CvM}, \Tc_{CvM}^* \big)$ weakly converges to a couple of identically distributed vectors 
when $n$ tends to the infinity, using the exchangeable (resp. nonparametric) bootstrap.
And the same result applies to~$\Tc_{KS}$.

\mds

\section{Application to the dependence between financial returns}
\label{financial_returns}

The data that we are considering is made up of three European stock indices (the French CAC40, the German DAX Performance Index and the Dutch Amsterdam Exchange index called AEX),
two US stock indices (the Dow Jones Index and the Nasdaq Composite Index), the Japan Nikkei 225 Index, two oil prices (the Brent Crude Oil and the West Texas Intermediate called WTI) and the Treasury Yield 5 Years (denoted as Treasury5Y). 
These variables are observed daily from the 16th September 2008 (the day following Lehman's bankruptcy) to the 11th August 2020.
We compute the returns of all these variables. We realize an ARMA-GARCH filtering on each marginal return using the R package \texttt{fGarch}~\cite{fGarchpackage} and choosing the order which minimizes the BIC. The nine final variables $X_{t,i},$ $i=1, \dots, 9$ are defined as the innovations of these processes.

\mds

Each variable $X_i,$ $i=1, \dots, 9$, can play the role of the conditioning variable $Z$. 
When this is the case, we consider boxes determined by their quantiles. This yields nine boxes, defined as follows:
$A_{1,i} := [q^{X_i}_{0\%}, q^{X_i}_{5\%}]$,
$A_{2,i} := [q^{X_i}_{0\%}, q^{X_i}_{10\%}]$,
$A_{3,i} := [q^{X_i}_{0\%}, q^{X_i}_{20\%}]$,
$A_{4,i} := [q^{X_i}_{5\%}, q^{X_i}_{10\%}]$,
$A_{5,i} := [q^{X_i}_{20\%}, q^{X_i}_{80\%}]$,
$A_{6,i} := [q^{X_i}_{80\%}, q^{X_i}_{100\%}]$,
$A_{7,i} := [q^{X_i}_{90\%}, q^{X_i}_{100\%}]$,
$A_{8,i} := [q^{X_i}_{90\%}, q^{X_i}_{95\%}]$,
$A_{9,i} := [q^{X_i}_{95\%}, q^{X_i}_{100\%}]$.
In the following, we always consider conditioning by one variable only.

\mds

Our measure of conditional dependence will be (conditional) Kendall's tau. Because of the high number of triplets (i.e. couples $(X_i,X_j)$ given $X_k$ belongs to some subset), we do not consider every combination of conditioned and conditioning variable, but only report a few relevant ones.
Figure~\ref{fig:CKT_EuroIndexes} is devoted to the dependence between European indices. 
Figure~\ref{fig:CKT_US-EuroIndexes} is related to the dependence between European indices and the Dow Jones.
Figure~\ref{fig:CKT_US_Nikkei} deals with dependencies between the Dow Jones and the Nikkei indices.
Dependencies between US indices appear in Figure~\ref{fig:CKT_US_Indexes}. 
In all figures, the dotted line represents the unconditional Kendall's tau of the considered pair of variables.

\mds

Note that $[q^{X_i}_{0\%}, q^{X_i}_{100\%}] = A_{3,i} \sqcup A_{5,i} \sqcup A_{7,i}$, where $\sqcup$ denotes disjoint union. Nevertheless, the unconditional Kendall's tau $\tau_{X_1, X_2}$ cannot be decomposed (and then deduced) using only the conditional Kendall's taus $\tau_{X_1, X_2 | X_i \in A_{3,i}}$, $\tau_{X_1, X_2 | X_i \in A_{5,i}}$ and $\tau_{X_1, X_2 | X_i \in A_{7,i}}$. Indeed, 
\begin{eqnarray}
\lefteqn{    \tau_{X_1, X_2} = \int \Big(
    \1 \big( (x_{1,1}-x_{2,1})(x_{1,2}-x_{2,2}) > 0 \big)   \nonumber }\\
    &-& \1 \big( (x_{1,1}-x_{2,1})(x_{1,2}-x_{2,2}) < 0 \big) \Big) \, d\Pb_{\X}(\x_1) \, d\Pb_{\X}(\x_2).
    \label{eq:def:uncondKT}
\end{eqnarray}
Formally, we can decompose the probability measure $\Pb_{\X}(B)$ as
$ \sum_{k \in \{3,5,7\}} \Pb(\X\in B | X_i \in A_{k,i}) \Pb(X_i \in A_{k,i}) $ for any borelian $B$. 
Expanding in~(\ref{eq:def:uncondKT}), we indeed get terms such as the conditional Kendall's tau $\tau_{X_1, X_2 | X_i \in A_{k,i}}$, but also ``co-Kendall's taus'' that involve integrals with respect to some measures $\Pb(\cdot | X_i \in A_{k,i}) \otimes \Pb(\cdot | X_i \in A_{k',i})$, $k \neq k'$. Therefore, it is possible that all conditional Kendall's taus are strictly smaller (or larger) than the corresponding unconditional Kendall's tau. This is indeed the case for the couple $X_1 = CAC40$, $X_2 = AEX$ and $Z = DAX$ (see Figure~\ref{fig:CKT_EuroIndexes}).

\mds

Many interesting features appear on such figures. For instance, the levels of dependence between two European stock indices (CAC40 and AEX, e.g.) are significantly varying depending on another European index (say, DAX). At the opposite, they are globally insensitive to shocks on the main US index or on oil prices. 
This illustrates the maturity of the integration of European equity markets.
Note that the strength of such moves matters: dependencies given average shocks (when $Z$ belongs to $A_4$ or $A_8$) are generally smaller than those in the case of extreme moves (when $Z$ belongs to $A_1$ or $A_9$, e.g.). This is a rather general feature for most figures.
Moreover, dependencies are most often larger when the conditioning events are related to ``bad news'' (negative shocks on stocks, sudden jumps for interest rates), compared to ''good news'' (the opposite events): see Figure~\ref{fig:CKT_US_Indexes}, that refers to the couple (Dow Jones, Nasdaq).
When the pairs of stock returns are related to two different countries, dependence levels are globally smaller on average, but this does not preclude significant variations knowing another financial variable belongs to some range of values. Therefore, when Treasuries are strongly rising, the dependence between Dow Jones and Nikkei can become negative - an unusual value - although it is positive unconditionally.

\begin{figure}[htbp]
    \makebox[\textwidth][c]{\includegraphics[width=1.3\textwidth]{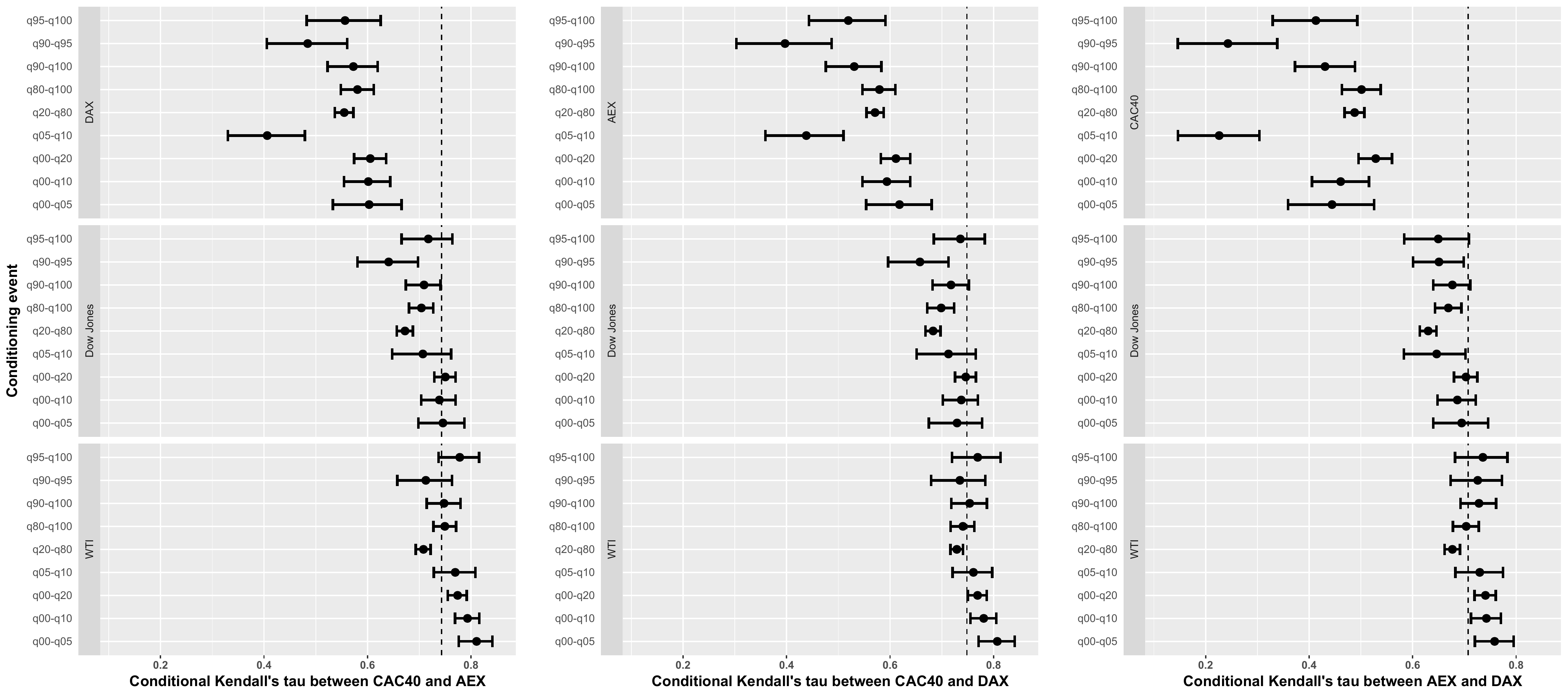}}
    \caption{Conditional Kendall's taus between pairs of European indexes.}
    \label{fig:CKT_EuroIndexes}
\end{figure}

\begin{figure}[htbp]
    \centering
    \includegraphics[width=0.9\textwidth]{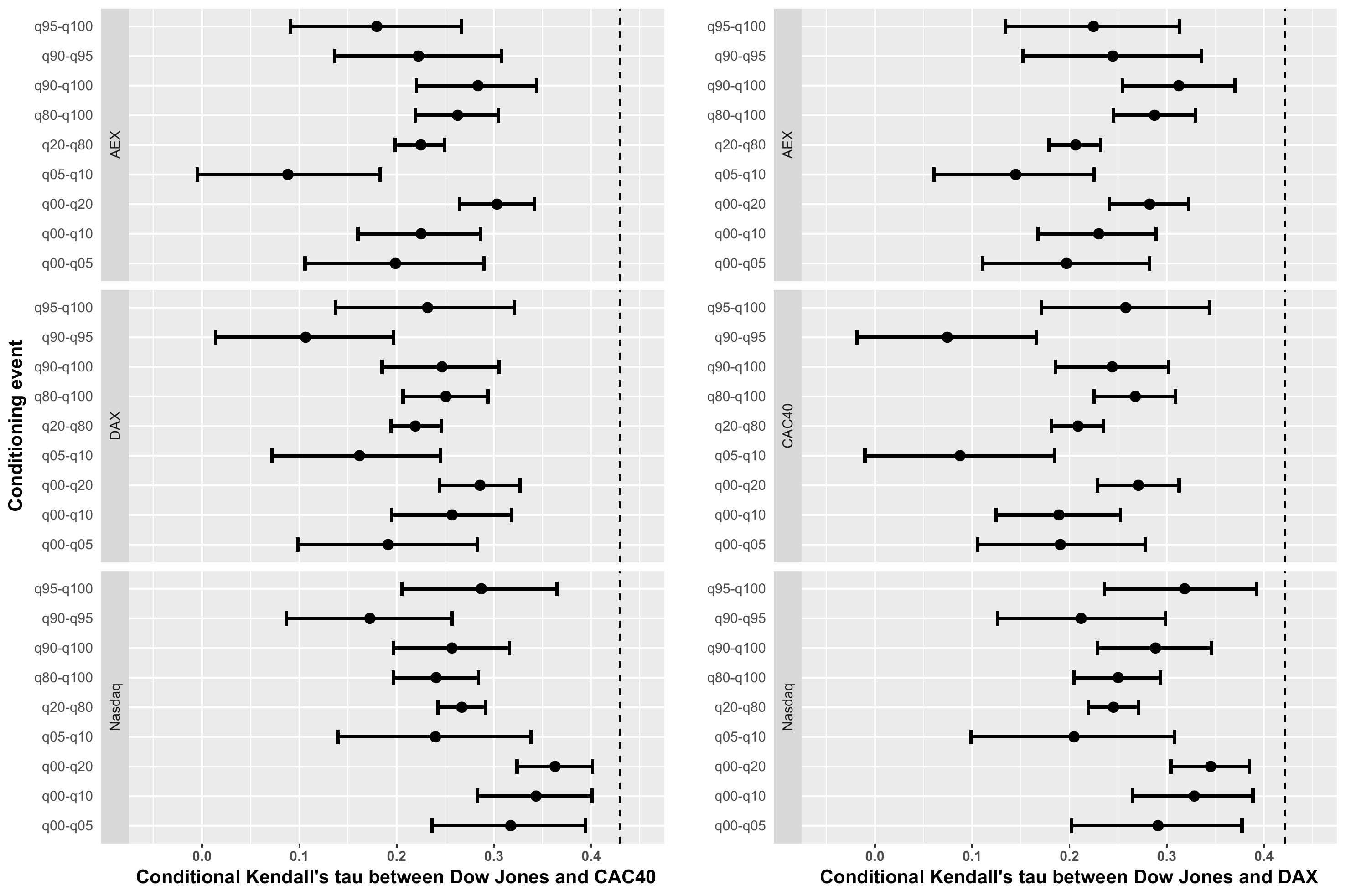}
    \caption{Conditional Kendall's tau between the Dow Jones Index and European indexes.}
    \label{fig:CKT_US-EuroIndexes}
\end{figure}

\begin{figure}[htbp]
    \centering
    \includegraphics[width=0.9\textwidth]{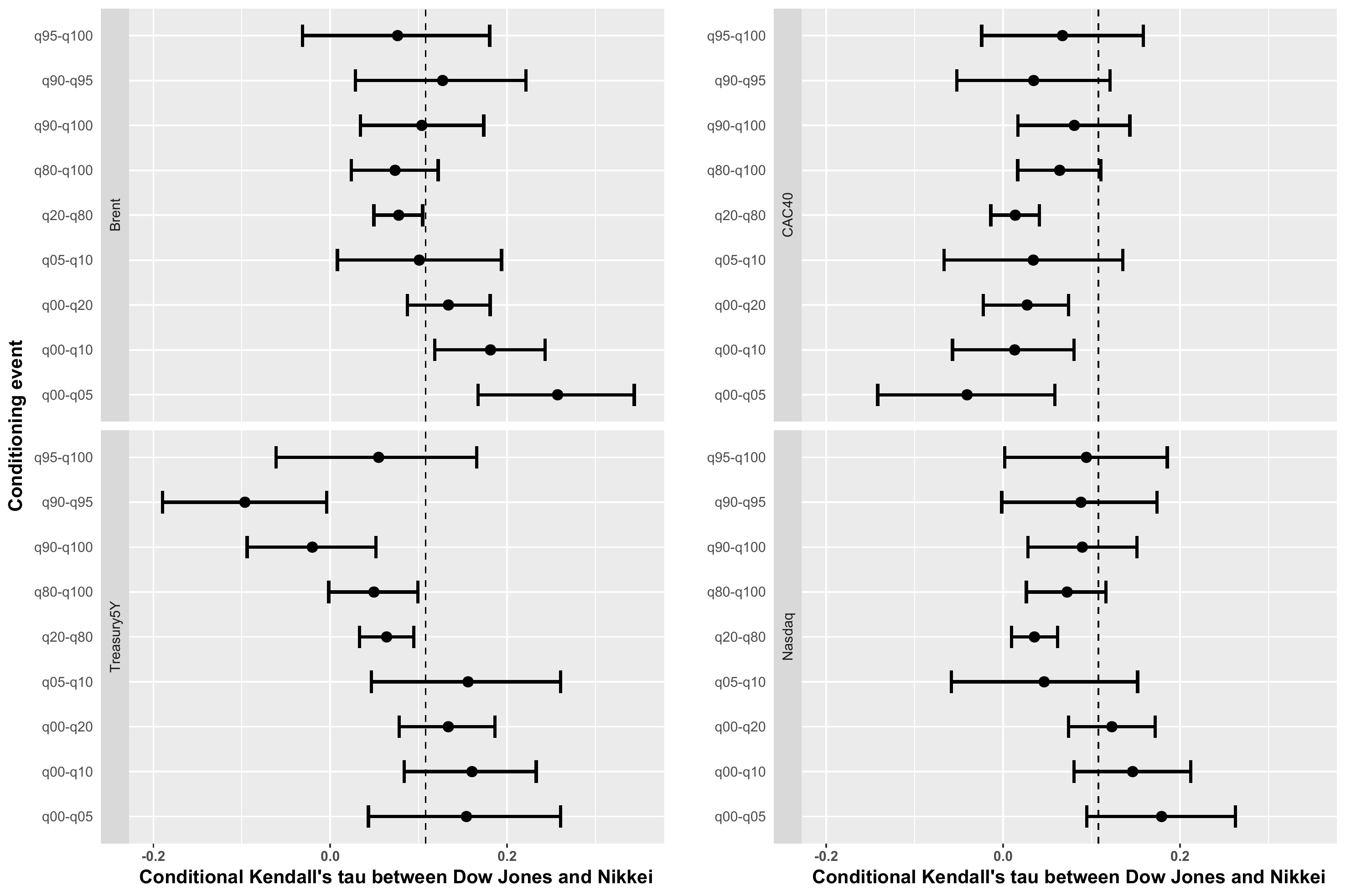}
    \caption{Conditional Kendall's tau between the Dow Jones Index and the Nikkei.}
    \label{fig:CKT_US_Nikkei}
\end{figure}

\begin{figure}[htbp]
    \centering
    \includegraphics[width=0.9\textwidth]{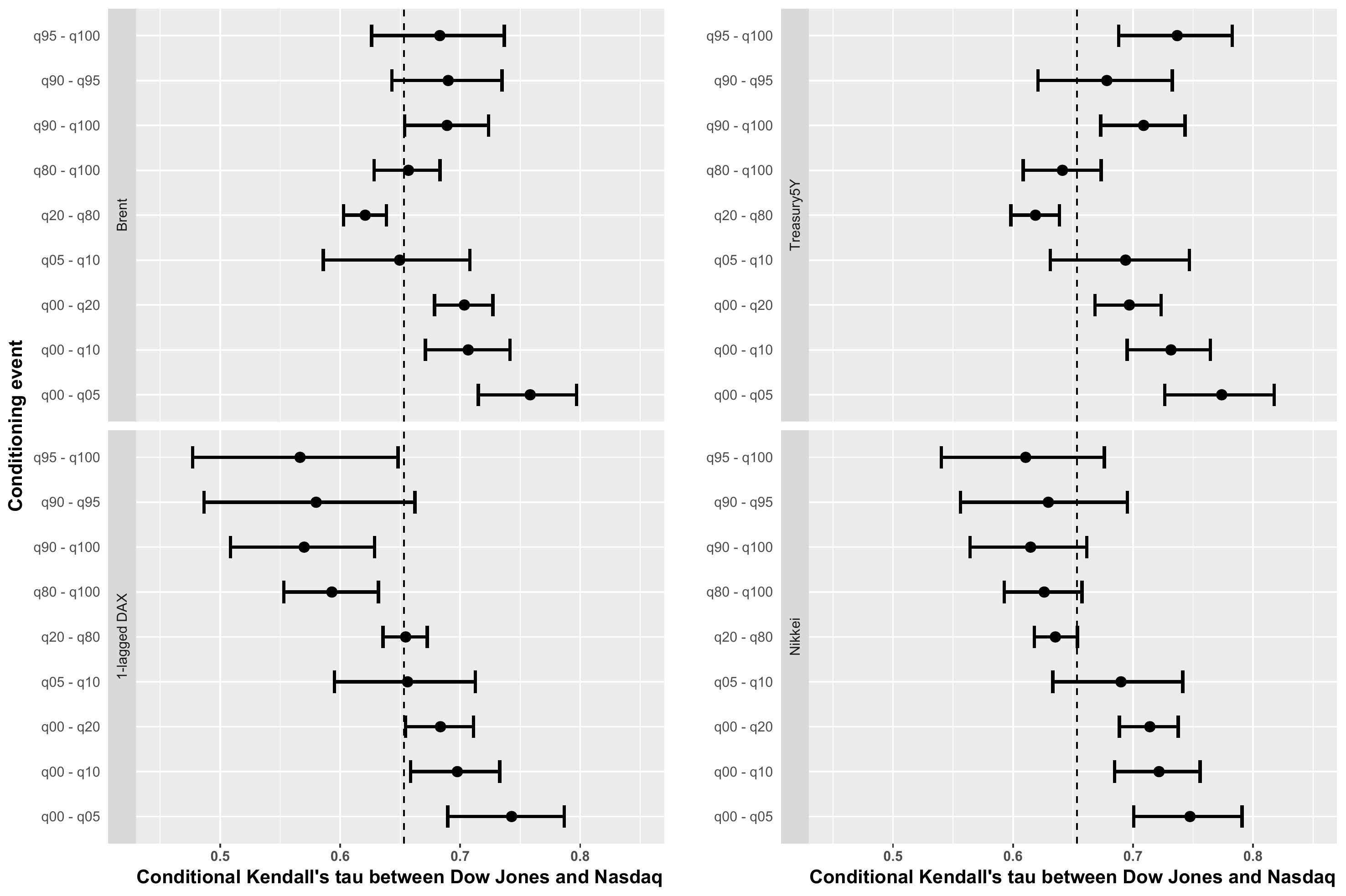}
    \caption{Conditional Kendall's tau between the Dow Jones Index and the Nasdaq Composite Index.}
    \label{fig:CKT_US_Indexes}
\end{figure}

\section{Conclusion}

We have made several contributions to the theory of the weak convergence of empirical copula processes, their associated bootstrap schemes and multivariate dependence measures.
Now, all these concepts and results are stated not only for usual copulas but for conditional copulas too, i.e. for the copula of $\X$ knowing that some vector of covariates $\Z$ (that may be equal to $\X$) belongs to one or several borelian subsets. We only require that the probabilities of the latter events are nonzero. Therefore, we do not deal with pointwise conditioning events as
$A=\{\Z=\z\}$ for some particular vector $\z$. But the main advantage of working with $\Z$-subsets instead of singletons is to avoid the curse of dimension that rapidly appears when the dimension of $\Z$ is larger than three.

\mds

Once we have proved the weak convergence of the conditional empirical copula process $\hat\Cb_n(\cdot |A)$, possibly with multiple borelian subsets $A_j$, the inference and testing of copula models becomes relatively easy. An interesting avenue for further research will be to use our results to build convenient discretizations of the covariate space (the space of our so-called random vectors $\Z$). There is a need to find efficient algorithms and statistical procedures to build a partition of $\Rb^q$ with borelian subsets $A_j$, so that the dependencies across the components of $\X$ are ``similar'' when $\Z$ belongs to one of theses subsets, but as different as possible from box to box: ``maximum homogeneity intra, maximum heterogeneity inter''. 
A constructive tree-based approach should be feasible, as proposed in~\cite{kurzspanhel2017} in the case of vine modeling.

\bigskip 

\noindent
{\bf Acknowledgements}
\noindent

\mds
%The authors are grateful for helpful discussions with the participants at the 7th Brazilian conference SMIF (Maresias, March 2020). 
Jean-David Fermanian's work has been supported by the labex Ecodec (reference project ANR-11-LABEX-0047).

\bigskip

\appendix

\section{Proofs}

\subsection{Proof of Theorem~\ref{WeakConvCopProcess}}
\label{sec:proof:WeakConvCopProcess}

Let us introduce the vector of (unobservable) empirical quantiles
$$ \v_n(\u):=\big( G_{n,1}^{-1}(u_1|A),\ldots, G_{n,p}^{-1}(u_p|A) \big).$$
Then, note that $\bar D_n(\u,A)=D_n(\v_n(\u),A)$. As in~\cite{segers}, let us decompose
\begin{eqnarray}
\lefteqn{\bar\Db_n(\u,A) = \sqrt{n}\big( \bar D_n-D \big)(\u,A)          \nonumber  }\\
&=&  \sqrt{n}\big\{ D_n(\v_n(\u),A) - D(\v_n(\u),A) \big\} + \sqrt{n}\big\{ D(\v_n(\u),A) - D(\u,A) \big\} \nonumber \\
&=&  \Db_n(\v_n(\u),A) + \sqrt{n}\big\{ D(\v_n(\u),A) - D(\u,A) \big\} .
\label{decompos_barDbn}
\end{eqnarray}
As a usual empirical process, $\Db_n(\cdot,A)$ weakly tends to a Gaussian process in $\ell^{\infty}([0,1]^p)$, here the Brownian bridge $\Bb(\cdot,A)$, defined in Corollary~\ref{cor_weak_convergence_Dn}.
In particular, it is equicontinuous.
Note that $n_A$ tends to the infinity a.s. when $n$ tends to the infinity.
Then, $\sup_{u \in [0,1]}|(G_{n,k})^{-1}(u|A) - u| $ tends to zero a.s. for every $k$, when $n$ (and then $n_A$) tends to the infinity.
Therefore, the equicontinuity of $\Db_n$ implies
\begin{equation}
    \sup_{\u\in [0,1]^p}
    \big| \Db_n(\v_n(\u),A) - \Db_n(\u,A) \big| \stackrel{p}{\longrightarrow} 0, \label{Equiv_Term1_Dbn}
 \end{equation}
when $n\rightarrow \infty$.

\mds
Moreover, fix $\u\in [0,1]^p$ and define $w(t)=\u + t\{ \v_n(\u)-\u\}$ for any $t\in [0,1]$. By the mean value theorem, there exists $t^*(\u)=:t^*\in [0,1]$ s.t.
$$ \sqrt{n}\big\{ D(\v_n(\u),A) - D(\u,A) \big\} = \sum_{k=1}^p \partial_k D(w(t^*),A) \sqrt{n}\big\{ G_{n,k}^{-1}(u_k| A) - u_k  \big\}. $$
The latter identity is true whatever the values of $\u\in [0,1]^p$, even if one of its components is zero (see the discussion in~\cite{segers}, p.769).
Denote by $\e_k$ the unit vector in $\Rb^p$ corresponding to the $k$-th component.
For every $\u\in[0,1]^p$ and $t\geq 0$, we have, with obvious notations,
\begin{align*}
    0 &\leq \big( D(\u + t\e_k, A)- D(\u,A) \big)/t
    = \Pb\big(\U_{-k}^A \leq \u_{-k}, U^A_k\in [u_k,u_k+t], \X_J\in A\big)/t \\
    &\leq \Pb \big(U^A_k\in [u_k,u_k+t], \X_J\in A\big)/t
    \leq \Pb \big(U^A_k\in [u_k,u_k+t]
    \,|\, \X_J\in A\big) p_A/t =p_A.
\end{align*}
We deduce
$$ \lim\sup_{t\rightarrow 0^+} 
\big| D(\u + t\e_k, A)- D(\u,A) \big|/t \leq p_A,$$
and then $\sup_{\u\in [0,1]^p} |\partial_k D(\u,A)| \leq 1$, when the latter partial derivative exists.

\mds
Due to the Bahadur-Kiefer theorems (see~\cite{shorackwellner}, chapter 15), it is known that
$$\sup_{u\in [0,1]} \big| \sqrt{n_A}(G_{n,k}^{-1}(u|A) - u) + \alpha_{n,k}(u|A) \big|=o_P(1),$$
for every $k$, when $n_A$ tends to the infinity. We deduce
\begin{equation*}
 \sup_{\u\in [0,1]^p}
 \bigg| \sqrt{n}\big\{ D(\v_n(\u),A) - D(\u,A) \big\} + \big(\frac{n}{n_A}\big)^{1/2}\sum_{k=1}^p
\partial_k D\big(\u + t^*\{ \v_n(\u)-\u\},A\big) \alpha_{n,k}(u_k|A)
\bigg| 
% \label{Equiv_Term2_Dbn}
\end{equation*}
tends to zero in probability, as $n$ tends to the infinity.

\mds
By adapting the end of the proof of Proposition 3.1. in~\cite{segers}, we easily prove that
\begin{equation}
\sup_{\u\in [0,1]^p} \big| \partial_k D\big(\u + t^*\{ \v_n(\u)-\u\},A\big) - \partial_k D\big(\u,A\big) \big| \times |\alpha_{n,k}(u_k | A) | =o_P(1),
\label{residual1}
\end{equation}
Moreover, with obvious notations,
\begin{eqnarray}
 \alpha_{n,k}(u_k|A)&=&\sqrt{n_A}\Big(\frac{D_n\big((u_k,\1_{-k}),A\big)}{\hat p_A} - u_k \Big)
= \frac{\sqrt{n}}{\sqrt{\hat p_A}}\Big(D_n\big((u_k,\1_{-k}),A\big) - u_k \hat p_A\Big)  \nonumber \\
&=& \frac{\sqrt{n}}{\sqrt{\hat p_A}}\Big( D_n\big((u_k,\1_{-k}),A\big) - u_k p_A + u_k (p_A - \hat p_A)\Big) \nonumber \\
&=& \frac{1}{\sqrt{p_A}}\Big( \Db_n\big((u_k,\1_{-k}),A\big) - u_k \Db_n(\1,A)\Big) + o_P(1),
\label{decompos_alphank}
\end{eqnarray}
since $D\big((u_k,\1_{-k}),A\big)=u_k p_A$ for every $k$ and $u_k\in [0,1]$.
Equations~(\ref{residual1}) and~(\ref{decompos_alphank})  yield
\begin{equation}
 \sup_{\u\in [0,1]^p}
 \bigg| \sqrt{n}\big\{ D(\v_n(\u),A) - D(\u,A) \big\} +\sum_{k=1}^p
\frac{ \partial_k D\big(\u,A\big)}{p_A} \Big( \Db_n\big((u_k,\1_{-k}),A\big) - u_k \Db_n(\1,A)\Big) \bigg| \stackrel{p}{\longrightarrow} 0,
\label{Equiv_Term2_Dbn}
\end{equation}
when $n\rightarrow\infty$. Finally, Equations~(\ref{decompos_barDbn}),~(\ref{Equiv_Term1_Dbn}) and~(\ref{Equiv_Term2_Dbn}) conclude the proof.
$\Box$

\mds

\subsection{Proof of Theorem~\ref{validity_boot_Dcn}}
\label{proof:validity_boot_Dcn}

%Let $\hat\Pb_n^*:= n^{-1}\sum_{i=1}^n W_{n,i}\delta_{\X_i}$ be the weighted %empirical measure and the
Let $\Pb_n$ be the empirical measure associated to $(X_i,Z_i)_{i=1,\ldots,n}$.
Set the weighted bootstrap empirical process
$\hat \Vb_n:=n^{-1/2}\sum_{i=1}^n\big(W_{n,i}-\bar W_n \big)\delta_{(\X_i,\Z_i)}$.
For every $\u\in [0,1]^p$ and $\y\in \Rb^p$, denote by $g_{n,\u}$, $g_{\u}$ and $g_{\y}$ the maps from $\Rb^{p}\times \Rb^q$ to $\Rb$ defined by
$$g_{n,\u}:(\x,\z) \mapsto \1\big( x_1 \leq F_{n,1}^{-1}(u_1 |A),\ldots,x_{p} \leq F_{n,p}^{-1}(u_p |A),\z\in A \big),$$
$$g_{\u}:(\x,\z) \mapsto \1\big( x_1 \leq F_{1}^{-1}(u_1 |A),\ldots,x_{p} \leq F_{p}^{-1}(u_p |A),\z\in A \big),$$
$$g_{\y}:(\x,\z) \mapsto \1\big( x_1 \leq y_1,\ldots,x_{p} \leq y_p,\z\in A\big).$$
The latter functions implicitly depend on the borelian subset $A$.
Set the classes of functions $\Gc:=\{g_{\u}:\u\in [0,1]^p\}$, $\Gc_n:=\{g_{n,\u}:\u\in [0,1]^p\}$ and
$\Gc_0:=\{g_{\y}:\y\in \Rb^p\}$.
Note that $\Gc_n$ and $\Gc$ are subsets of $\Gc_0$ and that
$\Db^*_n (\u,A)= \int g_{n,\u}(\x,\z) \hat \Vb_n(d\x,d\z)=\hat \Vb_n(g_{n,\u})$.
Moreover, with some usual change of variables, we have
\begin{eqnarray*}
\lefteqn{  \| g_{n,\u}- g_{\u} \|^2_{L_2(P)} = \int (g_{n,\u} - g_{\u})^2(\x,\z)\, \Pb_{(\X,\Z)}(d\x,d\z) }\\
&\leq &\sum_{k=1}^p \Pb\big(|X_k - F_k^{-1}(u_k|A)|\leq |F_{n,k}^{-1}(u_k|A) - F_k^{-1}(u_k|A)| \big).
%&\leq &\sum_{k=1}^p \Pb\big(|U_k - u_k|\leq |G_{n,k}^{-1}(u_k|A) - u_k|\big).
\end{eqnarray*}
For every $k\in \{1,\ldots,p\}$, we have
$$\sup_{u\in [0,1]}|F_{n,k}^{-1}(u|A) - F_k^{-1}(u|A)| = \sup_{v\in [0,1]}|G_{n,k}^{-1}(v|A) - v|,$$ 
that tends to zero a.s. (see~\cite{shorackwellner}, Chapter 13). This yields.
$ \sup_{\u\in [0,1]^p } \| g_{n,\u}- g_{\u} \|^2_{L_2(P)} =o_P(1).$
Since the process $\hat\Vb_n$ is weakly convergent in $\ell^{\infty}(\Gc_0)$ (Theorem 3.6.13 in~\cite{VVW}), it is equicontinuous and then
$\sup_{\u}| \hat \Vb_n(g_{n,\u}) - \hat \Vb_n(g_{\u})| = o_{P}(1)$. Therefore, the weak limit of $\Db^*_n(\cdot,A)$ on $\ell^{\infty}([0,1]^p)$ is the same as
the weak limit of $\hat \Vb_n$ on $\ell^{\infty}(\Gc)$  (also denoted $\ell^{\infty}([0,1]^p)$).

\mds
Since $\Gc$ is Donsker, Theorem 3.6.13 in~\cite{VVW} yields
$$ \sup_{h \in BL_1(\Gc)} | \Eb_{\W}[h(\hat \Vb_n)]- \Eb [h\big(\Bb(\cdot,A)\big)]    | \stackrel{P*}{\longrightarrow} 0.$$
%$$ \sup_{h \in BL_1(\Gc_n)} | \Eb_{\W}[h(\hat \Vb_n)]- %\Eb\Big[h\big(\Bb(\cdot,A)\big)\Big]    | \stackrel{P*}{\longrightarrow} 0.$$
Here, $BL_1(\Gc)$ denotes the set of functions $h:\ell^{\infty}(\Gc) \rightarrow [0,1]$ s.t.
$|h(T_1) - h(T_2|\leq \sup_{f\in \Gc}| T_1(f) - T_2(f)|$.
Moreover, due to the weak convergence of $\Pb_n$ in $\ell^\infty(\Gc)$, 
$$ \sup_{h \in BL_1(\Gc)} | \Eb_{\W}[h(\Pb_n)]- \Eb [h\big(\Bb(\cdot,A)\big)]    | \stackrel{P*}{\longrightarrow} 0.$$
Thus, the $BL_1(\Gc)$ bounded Lipschitz metric between the measures $\hat \Vb_n$ and $\Pb_n$ tends to zero in outer probability. 
Apply Lemma 2.2 in~\cite{KojaBuecher} (equivalence between the items (b) and (c)) to obtain the stated result. $\Box$

\mds

\subsection{Proof of Lemma~\ref{Hadamard_diff_rho}}
\label{sec:proof:Hadamard_diff_rho}

Let $BV_{M,p}$ be the space of right-continuous functions that are of bounded variation on $[0,1]^p$ in the sense of Hardy-Krause (see~\cite{RWZ} and the references therein, for instance), and whose total variation is bounded by a constant $M>1$, as in~\cite{VVW}, Lemma 3.9.17.
We can define $\Psi_{K,K'}$ on the space $BV_{M,p}$ by Equation~(\ref{def_map_Psi}).
Let $(t_n)$ be a sequence of nonzero real numbers, $t_n\rightarrow 0$ when $n\rightarrow 0$.
Consider a sequence $(h_n)$ of functions from $[0,1]^p$ to $\Rb$ s.t. $C+ t_n h_n$ belongs to $BV_{M,p}$ and $h_n$ tends to a continuous function $h$ in sup-norm.
Let us expand
\begin{align*}
    &\big\{\Psi_{K,K'}(C+t_n h_n)-\Psi_{K,K'}(C) \big\}/t_n - \Psi_{K,K'}'(C)(h) \\
    &= \int  \psi(\u)\, (h_n - h)_K(\u_K ) C_{K'}(d\u_{K'} )\, d\u_{I\setminus K'} 
    + \int \psi(\u)\, C_{K}(\u_K ) (h_n - h)_{K'}(d\u_{K'} )\, d\u_{I\setminus K'} \\
    &+ t_n\int \psi(\u)\, (h_{n}-h)_K(\u_K ) h_{n,K'}(d\u_{K'} )\, d\u_{I\setminus K'} 
    + t_n\int \psi(\u)\, h_{K}(\u_K ) h_{n,K'}(d\u_{K'} )\, d\u_{I\setminus K'} \\
    &=: T_{n,1} + T_{n,2} + T_{n,3}+T_{n,4}.
\end{align*}
If $K=\emptyset$ then $T_{1,n}=0$. Otherwise, since $\| h_n - h\|_\infty$ tends to zero and
$$\int |\psi(\u)\, C_{K'}(d\u_{K'} ) |\, d\u_{I\setminus K'} < \|\psi\|_\infty 
\int | C_{K'}(d\u_{K'} ) |\, d\u_{I\setminus K'}  <+\infty,$$
we get $T_{1,n}=o(1)$.

\mds
Moreover, if $K'=\emptyset$, then $T_{2,n}=0$. Otherwise,
$T_{2,n}$ can always be defined by an integration by parts formula (Theorem 15 in~\cite{RWZ}) that involves a finite number of terms as
$$ \int (h_n - h)_{K'}(\u_{L}-,\a_{K'\setminus L}) \chi(d\u_{L}, \a_{K'\setminus L}),$$
for some subsets of indices $L\subset K'$ and some vectors $\a_{K'\setminus L}$, with obvious vectorial notations.
Since $\chi$ is of bounded variation and $\|h_n-h\|_\infty$ tends to zero, this yields $T_{2,n}=o(1)$.

\mds
If $K=\emptyset$, then $T_{3,n}=0$. Otherwise, note that the total variation of $C+t_n h_n$ is less than a constant $M$. Therefore, the total variation of $t_n h_n$ is less than $2M$. We deduce $T_{n,3}$ tends to zero when $n\rightarrow \infty$.

\mds
If $K'= \emptyset$, then $T_{4,n}=o(1)$ because $t_n\rightarrow 0$.
To tackle the last term when $K'\neq \emptyset$,
we use the uniform continuity of the function $\u\mapsto \psi(\u) h(\u_K)$ on the compact subset $[0,1]^p$, as in Lemma 1 in~\cite{gijbels2011}:
for every $\eps>0$, there exists a partition of $[0,1]^p$ with $q=q(\eps)$ disjoint hyper-rectangles $R_j=(\a,\b]$ such that the stepwise function
$ s_\eps(\u):=\sum_{j=1}^q c_j \1(\u\in R_j)$ satisfies
$$ \sup_{\u\in [0,1]^p} |  s_\eps(\u) - \psi(\u) h(\u_K)|<\eps.$$
Then, we obtain
\begin{eqnarray*}
\lefteqn{ |T_{4,n}|\leq
|t_n|\int | \psi(\u)\, h_{K}(\u_K ) - s_\eps(\u) | \times |h_{n,K'}(d\u_{K'})| \, d\u_{I\setminus K'}  }\\
&+&
 |t_n|\times | \sum_{j=1}^q c_j \int \1(\u\in R_j) h_n(d\u_{K'}) \, d\u_{I\setminus K'} |  \leq  2M \eps + |t_n|\times | \sum_{j=1}^q c_j  h_n\big((\a_{K'},\b_{K'}]\big) |.
\end{eqnarray*}
Note that $h_n$ is bounded by a constant because it is uniformly convergent to the continuous map $h$.
Therefore, $| \sum_{j=1}^q c_j  h_n(R_j) |\leq H  \sum_{j=1}^q |c_j|$ for some constant $H$.
Then, for a given $\eps>0$ and when $n$ is sufficiently large,
$|T_{4,n}|\leq (2M+1)\eps$. This means $T_{4,n}=o(1)$, concluding the proof. $\Box$

\section{Covariance function of $\Cb_\infty(\cdot|A)$}
\label{covariances}
For every $\u_1$ and $\u_2$ in $[0,1]^p$ and two borel subsets $A_1$ and $A_2$ in $\Rb^q$, denote 
$$ v_{1,2}(\u_1,\u_2):=\Eb\big[\Bb(\u_1,A_1) \Bb(\u_2,A_2) \big].$$
When $A_1=A_2$, this is the covariance function of the Gaussian process $\Bb(\cdot,A)$, as given in Corollary~\ref{cor_weak_convergence_Dn}.
In the general case, it is given in~(\ref{general_cov_B}).
The goal is here to calculate the covariance function of the limiting processes obtained in Theorem~\ref{weak_conv_emp_copula_process_sets} 
and~\ref{weak_conv_emp_copula_process_sets_vectorial}.
To lighten notations, simply write $v$ instead of $v_{1,2}$ and $p_k$ instead of $p_{A_k}$, $k\in \{1,2\}$.
By lengthly but simple calculations, we obtain 
\begin{eqnarray*}
\lefteqn{   \Eb\big[ \Cb_\infty(\u_1|A_1)  \Cb_\infty(\u_2|A_1)\big]= \frac{v(\u_1,\u_2)}{p_1 p_2}
+ \frac{D(\u_1,A_1)D(\u_2,A_2)}{p^1_1 p_2^2}v(\1,\1)   }\\
&+& \frac{1}{p_1^2 p_2^2} \sum_{k,l=1}^p 
\partial_k D(\u_1,A_1) \partial_l D(\u_2,A_2) 
\Big\{ v\big( (u_{1,k},\1_{-k}), (u_{2,l},\1_{-l})\big)
- u_{1,k} v\big( \1, (u_{2,l},\1_{-l})\big)   \\
&-& u_{2,l} v\big( (u_{1,k},\1_{-k}), \1\big) + u_{1,k}u_{2,l} v(\1,\1)  \Big\} \\
&-& \frac{1}{p^2_1 p_2} \sum_{k=1}^p  \partial_k D(\u_1,A_1)  \Big\{ v\big( (u_{1,k},\1_{-k}), \u_2\big) - u_{1,k} v\big( \1, \u_2\big) \Big\} \\
&-& \frac{1}{p^2_2 p_1} \sum_{l=1}^p  \partial_l D(\u_2,A_2) \Big\{ v\big(\u_1, (u_{2,l},\1_{-k})\big) - u_{2,l} v\big( \u_1,\1\big) \Big\} \\
&-& \frac{1}{p^2_1 p_2} D(\u_1,A)v(\1,\u_2) - \frac{1}{p^2_2 p_1}  D(\u_2,A)v_A(\u_1,\1)       \\
&+& \frac{1}{p^2_1 p_2^2} \sum_{k=1}^p \partial_k D(\u_1,A_1) D(\u_2,A_2) \Big\{ v((u_{1,k},\1_{-k}),\1) - u_{1,k} v(\1,\1)   \Big\}  \\
&+& \frac{1}{p^2_1 p_2^2} \sum_{l=1}^p \partial_l D(\u_2,A_2) D(\u_1,A_1) \Big\{ v(\1, (u_{2,l},\1_{-k})) - u_{2,l} v(\1,\1)   \Big\}.
\end{eqnarray*}

For a given subset $A_1=A_2=A$, we get the covariance map of the limiting process $\Cb_\infty(\u|A)$.
In this case, we denote $v_{1,2}=v_A$. Then, check that $v_A(\1,\1) = p_A(1-p_A)$, $v_A((u_{1,k},\1_{-k}),\1) = p_A(1-p_A)u_{1,k}$,
and $v_A(\1, (u_{2,l},\1_{-k}))= p_A(1-p_A)u_{2,l}$, for any indices $k,l$ in $\{1,\ldots,p\}$.
Moreover, when $k\neq l$,
$$ v_A\big( (u_{1,k},\1_{-k}), (u_{2,l},\1_{-l})\big)= p_A C_{(X_k,X_l)}\big( u_k,u_l |A \big) - p_A^2 u_k u_l, $$
and $ v_A\big( (u_{1,k},\1_{-k}), (u_{2,k},\1_{-k})\big)= p_A \min(u_{1,k},u_{2,k}) - p_A^2 u_{1,k}u_{2,k}.$
Finally, $ v_A\big( \u_1,\1\big) = C(\u_1 |A) p_A(1-p_A),$ and $ v_A\big( \1,\u_2 \big) = C(\u_2 |A) p_A(1-p_A)$.

\mds

When there is a single subset $A$, the covariance $ \Eb\big[ \Cb_\infty(\u_1|A)  \Cb_\infty(\u_2|A)\big]$ can be easily estimated by replacing every unknown term above by an empirical counterpart.
Therefore, $v_A(\u_1,\u_2)$ may be estimated by
$$ \hat v_A(\u_1,\u_2) := \hat p_A \hat C_n( \u_1 \wedge \u_2 | \Z\in A) - \hat p_A^2 
\hat C_n( \u_1 | \Z\in A)\hat C_n( \u_2 | \Z\in A) .$$
Moreover any quantity $D(\u,A)$ and $\partial_k D(\u,A)$ would be estimated by $\bar D_n(\u,A)$ and $\widehat{\partial_k D}(\u,A)$ 
(see Equation~(\ref{def_estimated_derivatives})) respectively.

\mds

With multiple subsets $A_1,\ldots,A_m$, the task is slightly more complex. Indeed, recall that  
$$    \Eb\big[\Bb(\u_{1},A_1) \Bb(\u_{2},A_{2}) \big]
= \Pb( \U^{A_1} \leq \u_1, \U^{A_{2}} \leq \u_2, \Z\in A_1\cap A_{2}) - D( \u_1, A_1) D( \u_2, A_2) .$$

Reasoning as in Section~\ref{weak_convergence_section_single}, we easily 
see that any quantity $\Pb( \U^{A_j} \leq \u_j, \U^{A_{k}} \leq \u_k, \Z\in A_j\cap A_{k})$ can be empirically estimated by 
$\bar D_n(\u_j,\u_k,A_j,A_k)$ defined as 
\begin{eqnarray*}
\lefteqn{ \frac{1}{n} \sum_{i=1}^n \1 \big( U^A_{i,1}\leq G_{n,1}^{-1}(u_{i,1}|A_j) \wedge G_{n,1}^{-1}(u_{k,1}|A_k),\ldots,   }\\
&&\ldots, U^A_{i,p}\leq G_{n,p}^{-1}(u_{j,p}|A_j)\wedge G_{n,p}^{-1}(u_{k,p}|A_k), \Z_i\in  A_j\cap A_{k} \big) \\
&=& \frac{1}{n} \sum_{i=1}^n \1 \big( X_{i,1}\leq F_{n,1}^{-1}(u_{j,1}|A_j) \wedge F_{n,1}^{-1}(u_{k,1}|A_k),\ldots, \\
&&\ldots, X_{i,p}\leq F_{n,p}^{-1}(u_{j,p}|A_j)\wedge F_{n,p}^{-1}(u_{k,p}|A_k), \Z_i\in  A_j\cap A_{k} \big),
\end{eqnarray*}
for every $(\u_j,\u_k)\in [0,1]^{2p}$ and every indices $j,k$ in $\{1,\ldots,m\}$.
When $A_j$ and $A_k$ are disjoint, as in the case of partitions, the latter quantity is simply zero.
In every case, $\Eb\big[\Bb(\u_{1},A_1) \Bb(\u_{2},A_{2}) \big]$ is consistently estimated 
by $\bar D_n(\u_1,\u_2,A_1,A_2) - \bar D_n( \u_1, A_1) \bar D_n( \u_2, A_2)$.

\end{document}